\documentclass[12pt]{article}
\usepackage[usenames,dvipsnames,svgnames,table]{xcolor} 
\usepackage{kotex}
\usepackage{graphicx}
\usepackage{epsfig}
\usepackage{amssymb, amsmath, mathtools}
\usepackage{verbatim}
\usepackage{natbib}
\usepackage[affil-it]{authblk}
\usepackage{mathrsfs}
\usepackage{here}
\usepackage{makecell}
\usepackage{tikz}
\usepackage[shortlabels]{enumitem}
\usepackage{yfonts}
\usepackage{comment}
\usepackage{longtable}
\usepackage{tikz,amsmath, amssymb,bm,color}
\usetikzlibrary{circuits.logic.US,circuits.logic.IEC,fit}
\usepackage{algorithm2e}
\usepackage{multirow}

\usepackage[colorlinks=true, citecolor=black, urlcolor=black]{hyperref}
\usepackage[colorinlistoftodos, textsize=scriptsize]{todonotes} 

\newenvironment{proof}[1][Proof]{\begin{trivlist}
		\item[\hskip \labelsep {\bfseries #1}]}{\end{trivlist}}

\newtheorem{theorem}{Theorem}[section]
\newtheorem{lemma}[theorem]{Lemma}
\DeclareMathOperator*{\argmax}{argmax}
\newcommand{\half}{\frac{1}{2}}

\newcommand{\lra}{\longrightarrow}

\newcommand{\bbE}{{ \mathbb{E}}}
\newcommand{\calC}{\mathcal{C}}

\newcommand{\calP}{\mathcal{P}}

\newcommand{\calU}{\mathcal{U}}

\newcommand{\kommentar}[1]{}
\def\Real{\hbox{I\kern-.1667em\hbox{R}}}
\def\Reals{\hbox{\scriptsize I\kern-.1667em\hbox{R}}}

\newcommand{\bsu}{\boldsymbol{u}}
\newcommand{\bsv}{\boldsymbol{v}}

\newcommand{\bsw}{\boldsymbol{w}}

\newcommand{\bss}{\boldsymbol{s}}

\newcommand{\bsV}{\boldsymbol{V}}

\newcommand{\bsX}{\boldsymbol{X}}

\newcommand{\bsB}{\boldsymbol{B}}
\newcommand{\bsS}{\boldsymbol{S}}

\newcommand{\diag}{\mbox{diag}}
\newcommand{\bsSig}{\boldsymbol{\Sigma}}

\newcommand{\bfX}{{\bf X}}

\newcommand{\sg}{\Sigma}

\newcommand{\bea}{\begin{eqnarray*}}
	\newcommand{\eea}{\end{eqnarray*}}
\newcommand{\bean}{\begin{eqnarray}}
\newcommand{\eean}{\end{eqnarray}}
\newcommand{\benu}{\begin{enumerate}}
\newcommand{\eenu}{\end{enumerate}}

\newcommand{\bbR}{\mathbb{R}}

\newcommand{\bbP}{\mathbb{P}}

\newcommand{\cU}{\mathcal{U}}

\newcommand{\bsg}{\boldsymbol{\sigma}}

\title{The Beta-Mixture  Shrinkage Prior for Sparse Covariances with Posterior Minimax Rates\footnote{The first and second authors contributed equally to this work.}}
\author[1]{Kyoungjae Lee}
\author[1]{Seongil Jo}
\author[2]{Jaeyong Lee}
\affil[1]{Department of Statistics, Inha University}
\affil[2]{Department of Statistics, Seoul National University}

\begin{document}

\maketitle

\begin{abstract}
Statistical inference for sparse covariance matrices is crucial to reveal dependence structure of large multivariate data sets, but lacks scalable and theoretically supported Bayesian methods.
In this paper, we propose beta-mixture shrinkage prior, computationally more efficient than the spike and slab prior,  for sparse covariance matrices   and establish its minimax optimality in high-dimensional settings.
The proposed prior consists of beta-mixture shrinkage and gamma priors for off-diagonal and diagonal entries, respectively.
To ensure positive definiteness of the resulting covariance matrix, we further restrict the support of the prior to a subspace of positive definite matrices. 
We obtain the posterior convergence rate of the induced posterior under the Frobenius norm  and establish a minimax lower bound for sparse covariance matrices. The class of sparse covariance matrices for the minimax lower bound  considered in this paper is controlled by the number of nonzero off-diagonal elements and has more  intuitive appeal than those appeared in the literature. 
The obtained posterior convergence rate  coincides with the minimax lower bound  unless the true covariance matrix is extremely sparse.
In the simulation study, we show that the proposed method is computationally more efficient than competitors, while achieving comparable performance.
Advantages of the shrinkage prior are demonstrated based on two real data sets.
\end{abstract}



\section{Introduction}

Suppose $X_1, \ldots, X_n$ are independent $p$-dimensional random vectors from $N_p(0, \Sigma)$, the $p$-dimensional normal distribution with mean $0 \in \bbR^p$ and covariance $\Sigma\in \bbR^{p\times p}$. The covariance matrix of  a random vector is a fundamental parameter that expresses the marginal dependence structure of $X$.  
It is a basis for many multivariate statistical methods such as principal component analysis,  factor analysis, discriminant analysis, and linear regression, to name just a few. 
In this paper, we consider the Bayesian inference of covariance matrices when the dimension of observations, $p$, tends to infinity as the sample size, $n$, gets larger.  
We assume that most of the off-diagonal entries of a covariance matrix are zero, i.e., only few  pairs of variables have significant marginal dependences.
We propose the beta-mixture  shrinkage prior for sparse covariance matrix. 
The proposed methodology is computationally fast and attains the minimax posterior convergence rate under the Frobenius norm when the true covariance is not extremely sparse.

There are vast and rich frequentist literature on high-dimensional sparse covariance estimation. 
Various thresholding estimators \citep{bickel2008covariance,rothman2009generalized,cai2011adaptive,cai2012optimal} and lasso-type procedures \citep{bien2011sparse} have been proposed for simultaneously learning marginal dependence structures and estimating covariance matrices.
Among them, \cite{cai2011adaptive} proposed an adaptive thresholding estimator and proved that it achieves the minimax convergence rate for sparse covariance matrices by showing that the obtained rate coincides with the minimax lower bound obtained in \cite{cai2012optimal}.


For Bayesian inference of sparse covariance, the $G$-inverse Wishart prior \citep{silva2009hidden} is often used. 
The normalizing constant of the $G$-inverse Wishart prior is analytically intractable and needs the Monte Carlo method for its evaluation, which makes the posterior computation infeasible even when $p$ is moderately large.  
\cite{khare2011wishart} introduced a broad class of priors including $G$-inverse Wishart prior as a special case.
They provided a blocked Gibbs sampler to obtain samples from the resulting posterior, but the priors were only applicable to decomposable covariance graph models.
\cite{wang15} proposed stochastic search structure learning (SSSL). He placed the spike-and-slab prior for the off-diagonal elements of $\Sigma$ and put a modified version of the product of Bernoulli prior on the sparsity structure of the covariance. The modification of the product of Bernoulli priors allows one to avoid the intense normalizing constant computation, but the natural interpretation of the Bernoulli prior is lost.

In addition to the computational difficulties of the posterior of the sparse covariance, the Bayesian literature lacks the asymptotic properties of the posteriors.    
\cite{lee2018optimal} showed that the inverse Wishart prior achieves the minimax posterior convergence rate for a unstructured covariance matrix under the spectral norm. 
The posterior convergence rate under the Frobenius norm was also derived.
However, they focused only on unstructured covariance matrices and used the inverse Wishart prior, which is not suitable for sparse covariance matrices.
Neither \cite{khare2011wishart} nor \cite{wang15} established the asymptotic properties of the posteriors for sparse covariances.
Up to our knowledge, asymptotic properties of the posteriors induced by the priors for sparse covariance matrices have not been investigated yet.

In this paper, to fill the gap in the literature, we develop a scalable Bayesian inference for sparse covariance matrices supported by theoretical properties of posteriors. 
We propose a continuous shrinkage prior for the sparse covariance matrices.
Especially, the beta-mixture prior and the gamma prior are used for off-diagonal and diagonal entries of covariance matrices, respectively.
To ensure the positive definiteness of the resulting covariance matrix, we further restrict the prior on a class of positive definite matrices.
A blocked Gibbs sampler is derived to obtain posterior samples.

The advantage of the proposed method are as follows.  
First, this is the first Bayesian method for sparse covariance matrices with optimal minimax rate unless the true covariance is extremely sparse. 
We show the posterior convergence rate of the proposed prior under the Frobenius norm (Theorem \ref{thm:conv_rate}). 
 We also derive a lower bound of the minimax rate for sparse covariance matrices (Theorem \ref{thm:lower}) with restriction only on the total number of nonzero off-diagonal entries, which differentiates the obtained result from the results in the literature assuming column-wise sparsity and has more intuitive appeal. 
These results show that the obtained posterior convergence rate is the minimax rate except extremely sparse cases. 
Second, the proposed method is computationally efficient.  
We  compare computational efficiency  of the shrinkage prior and the SSSL \citep{wang15},  and find that the proposed shrinkage prior has almost twice as many effective sample size as the SSSL.
This implies that the posterior sampling of the shrinkage prior  exhibits faster mixing than that of the SSSL.  

The paper is organized as follows. In Section \ref{sec:pre}, we describe the model, prior and the posterior computation. 
In Section \ref{sec:main}, we present the theoretical results including the asymptotic minimaxity.
The numerical studies and real data analysis are given in Section \ref{sec:simul}.
Concluding remarks are given in Section \ref{sec:disc}.

\section{Beta-Mixture Shrinkage Prior }\label{sec:pre}

\subsection{Notation}
Let $a_n$ and $b_n, n=1,2, \ldots$  be sequences of positive real numbers.  
We denote $a_n = O(b_n)$, or equivalently $a_n \lesssim b_n$, if $a_n/b_n \le C$ for some constant $C>0$.
We denote $a_n \asymp b_n$ if $a_n = O(b_n)$ and $b_n = O(a_n)$.
Furthermore, we denote $a_n = o(b_n)$, or equivalently $a_n \ll b_n$, if $a_n/b_n \lra 0$ as $n\to\infty$.
Let $\calC_p$ be the set of all $p\times p$ positive definite matrices. 
Let $A =(a_{ij})$ be a $p\times p$ matrix. 
We denote the minimum and maximum eigenvalues of $A$ by  $\lambda_{\min}(A)$ and $\lambda_{\max}(A)$,  respectively. 
The Frobenius norm of $A$ is defined by $\|A\|_F = (\sum_{i=1}^p \sum_{j=1}^p a_{ij}^2 )^{1/2}$.

\subsection{Prior   for sparse covariances}

Suppose we observe $n$ independent samples $\bfX_n = (X_1,\ldots, X_n)$ from the $p$-dimensional normal distribution:
\begin{equation}\label{model}
	X_i  \mid \sg \stackrel{iid}{\sim} N_p\left( 0 , \Sigma\right), \quad i = 1, \ldots, n,
\end{equation}
where $\sg \in \calC_p$.
We assume that the covariance matrix $\sg$ is $\ell_0$-sparse, i.e., most of off-diagonal entries of $\sg$ are  zero.
For Bayesian inference on $\sg$, we need to impose a prior distribution on a set of covariance matrices.
We first define a prior for $p\times p$ symmetric matrices and restrict it to the space of positive definite matrices.  
Let
\begin{eqnarray}
	\pi^u(\sigma_{jk} \mid \rho_{jk}) &= & N \Big(\sigma_{jk} \mid  0, \, \frac{\rho_{jk}}{1-\rho_{jk}} \tau_1^2 \Big) ,  \label{sig_jk} \\
	\pi^u(\rho_{jk}) & = & Beta(\rho_{jk} \mid a, b) ,\quad 1\le j < k  \le p ,  \label{rho_jk} \\
	\pi^u(\sigma_{jj}) & = & Gamma(\sigma_{jj} \mid c, d) , \quad j=1, \ldots, p  , \label{sig_jj}
\end{eqnarray}
for some positive constants $\tau_1, a, b, c, d$, where $Beta(a, b)$ is the beta distribution with parameters $a, b > 0$ and $gamma(c, d)$ is the beta distribution with shape parameter $c$ and rate parameter $d$.  The prior on symmetric matrix with positive diagonal elements is defined as 
\bean\label{sprior_init}
\pi^u(\Sigma) &=& \prod_{1\le j<k \le p} \pi^u(\sigma_{jk} \mid  \rho_{jk})\pi^u(\rho_{jk}) I(\sigma_{jk}= \sigma_{kj}) \prod_{j=1}^p \pi^u(\sigma_{jj}), 
\eean
where ``u'' stands for the unconstrained prior. 
Note that the marginal prior on $\sigma_{jk}$ is the half-Cauchy prior if we take $a=b=1/2$, which is one of the most popular shrinkage priors.

Other possible choices for the shrinkage prior of $\sigma_{jk}$ are  the horseshoe prior \citep{carvalho2010horseshoe}, the lasso prior \citep{park08},  the hyperlasso prior \citep{griffin11, griffin17} and the generalized double pareto (GDP) prior \citep{armagan13}. In this paper, we will focus on the half-Cauchy prior for the off-diagonal elements and $Gamma(1, \frac{\lambda}{2} ), \lambda > 0,$ and derive its theoretical properties. 

Now, we propose the shrinkage prior for sparse covariance matrices by restricting $\pi^u(\sg)$ to the subspace of positive definite matrices: 
\bean\label{sprior}
{\pi}( \Sigma ) &=& \frac{\pi^u(\Sigma) I (\Sigma \in \mathcal{U}(\tau) ) }{ \pi^u( \Sigma \in \mathcal{U}(\tau)  ) } ,
\eean
where 
\bea
\mathcal{U}(\tau) &=& \Big\{  \sg\in \calC_p:  \tau^{-1} \le \lambda_{\min}(\sg)\le \lambda_{\max}(\sg)\le \tau      \Big\} 
\eea
for some constant $\tau>1$.  
In this paper, we consider $\tau$ as a fixed constant to obtain desired asymptotic properties of posteriors.
However, in practice, one can use $\tau= \infty$, which results in $\calU(\tau) = \calC_p$.
Conditions on the hyperparameters will be specified in Section \ref{sec:main}, while practical suggestions will be given in Section \ref{sec:simul}.

\subsection{Comparison to the SSSL}

The shrinkage prior \eqref{sprior}  proposed in this paper and the SSSL proposed by \cite{wang15} use the gamma and exponential priors for the diagonal entries, $\sigma_{ii}, i=1,2, \ldots, p$,  of the covariance,  respectively. For the off-diagonal elements, $\sigma_{jk}, 1\leq  j \neq k \leq p$, \cite{wang15} used the continuous spike and slab prior, 
$$\pi^{u, W}(\sigma_{jk}) = (1-\pi) N(\sigma_{jk} \mid   0, \nu_0^2) + \pi N( \sigma_{jk}\mid 0, \nu_1^2) $$ for some constants $ 0<\nu_0 < \nu_1$ and $\pi \in (0,1)$, while we use the continuous beta-mixture shrinkage prior \eqref{sig_jk} and \eqref{rho_jk}.

In the spike and slab prior, the prior inclusion probability,  $\pi \in (0,1)$, reflects the prior belief  whether $\sigma_{jk}$ will be zero or not. 
Similarly to the beta-mixture  shrinkage prior \eqref{sprior}, \cite{wang15} proposed the  prior,  $\pi^{W}(\sigma_{jk})$, by restricting $\pi^{u, W}(\sigma_{jk})$ to the space of positive definite matrices.
Note that due to the unknown normalizing constant caused by the positive definiteness constraint, $\pi \in (0,1)$ is no longer the prior inclusion probability of the resulting prior $\pi^{W}(\sigma_{jk})$.

The main advantages of the beta-mixture prior over the SSSL are the theoretical guarantee and  computational efficiency. 
The proposed prior \eqref{sprior} achieves the minimax posterior convergence rate for sparse covariances under the Frobenius norm, which will be rigorously stated in Section \ref{sec:main}.
On the other hand, asymptotic properties of posteriors based on the SSSL have not been investigated yet.
Furthermore, based on the simulation studies in Section \ref{sec:simul}, we found that the finite sample performance of the proposed prior is comparable to that of the SSSL while achieving almost twice as many effective sample size. 



\subsection{Blocked Gibbs sampler}

We now provide a posterior sampling algorithm for our prior described in \eqref{sprior}. The algorithm is based on the blocked Gibbs sampler proposed by \cite{wang15}. To describe the algorithm,  as in Proposition 2 of \cite{wang15}, we consider the following partition of $\Sigma$, $\bsS = \bfX_n^T\bfX_n$ and $\bsV = (v_{jk}^2), ~v_{jk}^2 = v_{kj}^2 = \rho_{jk}\tau_1^2/(1-\rho_{jk})$ for $j < k$ and $v_{jk}^2 = 0$ for $ j = k$:
\begin{equation}\label{eq:partitions}
	\Sigma = \left(\begin{array}{cc}\Sigma_{11} & \bsg_{12} \\ \bsg_{12}^T & \sigma_{22} \end{array}\right), \quad 
	\bsS = \left(\begin{array}{cc}\bsS_{11} & \bss_{12} \\ \bss_{12}^T & s_{22} \end{array}\right), \quad
	\bsV = \left(\begin{array}{cc}\bsV_{11} & \bsv_{12} \\ \bsv_{12}^T & 0 \end{array}\right),
\end{equation}
where $\Sigma_{11} , \bsS_{11} , \bsV_{11}  \in \calC_{p-1}$, $\bsg_{12}, \bss_{12} ,\bsv_{12}  \in \bbR^{ (p-1)\times 1}$ and $\sigma_{22}, s_{22}  >0$,
and the change of variables:
\begin{equation}\label{eq:changev}
	\left(\bsg_{12}, \sigma_{22}\right) \rightarrow \left(\bsu = \bsg_{12}, v = \sigma_{22} - \bsg_{12}^T\Sigma_{11}^{-1}\bsg_{12}\right).	
\end{equation}
	
The posterior samples then are generated by iterating the following steps (for details, see the Appendix D):
\begin{itemize}
	\item For $\bsu$, 
	\bea
		\bsu \mid {\tt others} &\sim& \ N_{p-1} \left[\left\{\bsB + \diag(\bsv_{12}^{-1})\right\}^{-1}\bsw, \left\{\bsB + \diag(\bsv_{12}^{-1})\right\}^{-1}\right],
	\eea
	where $\bsB = \Sigma_{11}^{-1}\bsS_{11}\Sigma_{11}^{-1}v^{-1} + \lambda\Sigma_{11}^{-1}$ and $\bsw = \Sigma_{11}^{-1}\bss_{12}v^{-1}$.
	\item For $v$, 
	\bea
		v \mid {\tt others} &\sim& GIG \left(1 - n/2, \,\, \lambda, \,\, \bsu^T\Sigma_{11}^{-1}\bsS_{11}\Sigma_{11}^{-1}\bsu - 2s_{12}^T\Sigma_{11}^{-1}\bsu + s_{22}\right),
	\eea
	where $GIG(q, a, b)$ is the generalized inverse Gaussian distribution with the probability density function $f(x) \propto x^{q-1} e^{-(ax +b/x)/2} I(x>0)$.
	\item For $\rho_{jk} = 1 - 1/(1 + \phi_{jk})$, 
	\bea
	\psi_{jk} \mid {\tt others} &\sim& Gamma \left(a + b, \phi_{jk} + 1\right), \\
	\phi_{jk} \mid {\tt others} &\sim& GIG \left(a - 1/2, 2\psi_{jk}, \sigma_{jk}^2/\tau_1^2\right)  ,
	\eea
	where $Gamma(a,b)$ is the Gamma distribution the shape parameter $a$ and the rate parameter $b$.
\end{itemize}

\section{Posterior convergence rate}\label{sec:main}

In this section, we show that the beta-mixture shrinkage prior achieves the minimax rate under the Frobenius norm when $p = O(s_0)$ where $s_0$ is an upper bound for nonzero off-diagonal elements of the covariance matrix.  
Let $\sg_0 $ be the true covariance matrix.
For a given integer $0<s_0<p(p-1)$ and a real number $\tau_0 >1$, we define the parameter space 
\bean\label{para_sp}
\mathcal{U}(s_0, \tau_0)
&=& \Big\{ \Sigma \in \calC_p: |s(\Sigma)| \le s_0, \,\, \tau_0^{-1} \le \lambda_{\min}(\sg)\le \lambda_{\max}(\sg)\le \tau_0 \Big\} ,
\eean
where $|s(\sg)|$ is the number of nonzero off-diagonal entries in $\sg$.
To attain the desired asymptotic properties of posteriors, we introduce the following conditions.   

\begin{enumerate}[A1)]
\item[\bf (A1)] $\sg_0 \in \calU(s_0,\tau_0)$ for some integer $0<s_0<p(p-1)$ and constant $\tau_0 >1$. 
\item[\bf (A2)] $p \asymp n^{\beta}$ for some $0<\beta<1$.  
\item[\bf (A3)]  The hyperparameters satisfy $\tau   \ge \max(3, \tau_0)$, $\tau = O(1)$, $\lambda = O(1)$, 
$a=b=1/2$ and $\tau_1^2  \asymp 1 / (n p^4)$.
\end{enumerate}

Condition (A1) implies that the true covariance matrix is sparse and has eigenvalues bounded above as well as away from zero.  
The integer $s_0$ controls the sparsity of the true covariance matrix.
The bounded eigenvalue condition has been commonly used in high-dimensional matrix estimation literature including \cite{banerjee2015bayesian}, \cite{gao2015rate} and \cite{lee2019minimax}.
In this paper, the lower bound for the minimum eigenvalue is mainly used in Lemma \ref{KV.ineq} to convert $\|\sg_0^{-1} - \sg^{-1}\|_F$ to $\|\sg_0  - \sg \|_F$, while the upper bound for the maximum eigenvalue is required to ensure $\sg_0 \in \cU(\tau)$.

Condition (A2) says that the number of variables $p$ grows to infinity as $n\to\infty$, but at a slower rate than $n$.
In the literature, \cite{lam2009sparsistency} used a similar condition to obtain the convergence rates of  penalty estimators for  sparse  covariance matrices, and 
 \cite{liu2019empirical} used the same  condition to obtain the posterior convergence rate  for sparse precision matrices. 
This condition is inevitable to obtain the posterior convergence rate under the Frobenius norm if one uses the traditional techniques in \cite{ghosal2000convergence} which we use in this paper. 
In the seminal work of \cite{ghosal2000convergence}, they provided a sufficient condition for proving the posterior convergence rate for densities under the Hellinger metric, which is equivalent to the Frobenius norm for covariance matrices under the bounded eigenvalue condition (A1).
When the posterior is intractable, this is the standard way to find the posterior convergence rate.
One of necessary conditions in this result is that the posterior convergence rate should converge to zero as $n\to\infty$. Since the diagonal elements of the covariance are all nonzero, this condition requires the number of diagonal elements $p = o(n)$. This can be also seen from the minimax lower bound result, Theorem \ref{thm:lower}. Thus, if one use the techniques in \cite{ghosal2000convergence},  condition (A2) is required to prove Theorem \ref{thm:conv_rate}.

Condition (A3) gives a sufficient condition for hyperparameters to obtain the desired theoretical property of posteriors.
The choice $a=b=1/2$ implies that we use the half-Cauchy prior for the off-diagonal entries.
Note that $\tau_1^2$ is the global shrinkage parameter in \eqref{sig_jk}, thus condition (A3) means that the global shrinkage parameter should be sufficiently small.
This corresponds to assume a sufficiently small inclusion probability in spike and slab priors. See \cite{lee2019minimax} and \cite{martin2017empirical}.

For the asymptotic minimax rate of the shrinkage prior, we first show an upper bound of the minimax rate: Theorem \ref{thm:conv_rate} shows the posterior convergence rate of the proposed prior under the Frobenius norm.

\begin{theorem}\label{thm:conv_rate}
	Under model \eqref{model} and prior \eqref{sprior}, assume conditions (A1)--(A3) hold.
	If  $(p+s_0)\log p =o(n)$, as  $n\to\infty$ 
	\bea
	\pi \Big\{  \|\sg - \sg_0 \|_F^2 \ge M \frac{(p+s_0)\log p}{n}  \mid \bfX_n  \Big\}  &\lra& 0 ~ \text{in $\bbP_{0}$-probability}
	\eea
	for some large constant $M>0$. 
\end{theorem}

The condition $(p+s_0) \log p = o(n)$ relates $p$ and $n$ to $s_0$, the number of  nonzero off-diagonal elements of the true covariance. 
\cite{banerjee2015bayesian} and \cite{liu2019empirical} used the same condition to obtain the posterior convergence rate for sparse precision matrices.

The next theorem shows a minimax lower bound for covariance matrices, which coincides with the posterior convergence rate in Theorem \ref{thm:conv_rate} when $p = O(s_0)$.
In general, $p = O(s_0)$ holds unless $\sg_0$ is extremely sparse and  it holds  when the true covariance matrix has an autoregressive structure with order $1$, $AR(1)$.

\begin{theorem}\label{thm:lower}
	For given positive integer $s_0$ and real number $\tau_0 >1$, assume model \eqref{model} with $\Sigma_0 \in  \mathcal{U}(s_0, \tau_0)$. 
	If $s_0^2 (\log p)^3 = O(p^2 n)$ and $s_0^2 = O(p^{3-\epsilon})$ for some small constant $\epsilon>0$, 
	\bea
	\inf_{\hat{\sg}} \sup_{\sg_0 \in \calU(s_0, \tau_0) } \bbE_0 \| \hat{\sg} - \sg_0 \|_F^2  &\gtrsim& \frac{s_0 \log p}{n} \, I(s_0 > 3p) + \frac{p}{n} .
	\eea
\end{theorem}

\cite{cai2012optimal} proved that a modified thresholding estimator attains the minimax rate for sparse covariance matrices under the class of Bregman divergences.
They assumed sparsity for each column of the covariance matrix, which means that {\it each column} of $\sg_0$ has nonzero entries less than $s_0'$.
On the other hand, we assume that the nonzero entries of the $\sg_0$ is less than $s_0$.
Thus, our sparsity assumption on $\sg_0$ is much weaker than that of \cite{cai2012optimal}.
 Up to our knowledge, this is the first minimax lower bound result for sparse covariance matrices with restriction only on the total number of nonzero off-diagonal entries.
To establish the minimax rate, they assumed that $(s_0')^2 (\log p)^3 = O(n)$, which is roughly equivalent to $s_0^2 (\log p)^3 = O(p^2 n)$ in our notation.
It is easy to see that the minimax rate in \cite{cai2012optimal} coincides with the rate of the lower bound in Theorem \ref{thm:lower}.
Hence, Theorems \ref{thm:conv_rate} and \ref{thm:lower} imply that, even though we consider a larger parameter space than \cite{cai2012optimal}, the minimax rate is still unchanged and the proposed prior attains it.

\section{Simulation Study}\label{sec:simul}

\subsection{Synthetic data}
To assess the performance of our sparse covariance estimator, we carry out a simulation study using synthetic datasets generated from Gaussian distributions with zero means and the following two covariance structures $\Sigma = (\sigma_{jk})$.
\begin{itemize}
	\item[{\tt C1}.] Sparse covariance that mimics daily currency exchange rate return structure \citep{wang15}: 
	{\scriptsize
		$$
		\left(\begin{array}{rrrrrrrrrrrrr}
		0.239 & 0.117 & & & & & & 0.031 & & & & & \\
		0.117 & 1.554 & & & & & & & & & & & \\
		& & 0.362 & 0.002 & & & & & & & & \\
		& & 0.002 & 0.199 & 0.094 & & & & & & & \\
		& & & 0.094 & 0.349 & & & & & & & -0.036 \\
		& & & & & 0.295 & -0.229 & 0.002 & & & & \\
		& & & & & -0.229 & 0.715 & & & & & \\
		0.031 & & & & & 0.002 & & 0.164 & 0.112 & -0.028 & -0.008 & \\
		& & & & & & & 0.112 & 0.518 & -0.193 & -0.090 & \\
		& & & & & & & -0.028 & -0.193 & 0.379 & 0.167 & \\
		& & & & & & & -0.008 & -0.090 & 0.167 & 0.159 & \\
		& & & & -0.036 & & & & & & & 0.207
		\end{array}\right);
		$$
	} 
	\item[{\tt C2}.] Random structure: $\sigma_{jj} \sim Gamma(1,1), ~ j =1,2, \ldots, p$.  Sparse off-diagonal positions ($20\%$) are randomly selected.  Nonzero off-diagonal elements are generated from $Unif(0, \mu), ~ j \neq k$. 
\end{itemize}

In  case  {\tt C1},  following \cite{wang15} we take $p = 12$ and $n = 250$, and   in case  {\tt C2},  we consider $p = 50$ and $n = 50, 100$ and choose the range of  parameter  $\mu$ as $\{0.02, 0.1, 0.5, 1\}$ following  \cite{castillo20}.

We compute the root mean squared error ({\tt rmse}), $\|\hat{\Sigma} - \Sigma\|_F / p$, and the maximum norm ({\tt mnorm}), $\mbox{max}_{jk} |\hat{\sigma}_{jk} - \sigma_{jk}|$, and  report the average {\tt rmse} and {\tt mnorm}, and their standard errors over 50 replications. As competing methods, we consider the sample covariance (SampCov) and the SSSL estimator \citep{wang15}. 
We generate 5000 posterior samples after 5000 burn-in for  the proposed shrinkage model and the SSSL.

Tables \ref{tb:sim1} and \ref{tb:sim2} show the {\tt rmse} and {\tt mnorm} values for simulation models  with two covariance structures, {\tt C1} and {\tt C2}, respectively. 
Figure \ref{fig:ess1} renders the effective sample sizes (ESS),  equivalent  sample sizes when the independent sampling is done, of the posterior sampling for the  shrinkage prior and the SSSL in case {\tt C1}. 
From Tables \ref{tb:sim1} and \ref{tb:sim2}, we can see that  the proposed shrinkage  and the SSSL estimators are better than the sample covariance in all cases, and that the proposed shrinkage estimator performs better or at least  comparable to the  SSSL estimator, while the  posterior sampling algorithm of the shrinkage prior is more efficient than that of the SSSL in terms of ESS (Figure \ref{fig:ess1}). Additionally, the posterior sampling of the shrinkage prior  takes about 171 seconds per 1,000 ESS, but that of the SSSL takes 789 seconds with iMac Pro with 3 GHz 10-Core Intel Xeon processor. Finally, Table \ref{tb:sim2} shows that the continuous shrinkage prior produces more accurate estimates when signals are small, while the spike and slab prior can capture large signals more efficiently.

\begin{table}[!ht]
	\centering
	\caption{{\tt rmse} and {\tt mnorm} under the covariance  structure {\tt C1}.}
	\begin{tabular}{l|rrr}
		\hline\hline
		& Proposed & SSSL & SampCov \\ \hline
		{\tt rmse} & {\bf 0.020 (0.003)} &{\bf 0.020 (0.003)} & 0.028 (0.004) \\
		{\tt mnorm} & {\bf 0.114 (0.051)} & 0.120 (0.050) & 0.120 (0.054) \\
		\hline\hline
	\end{tabular}
	\label{tb:sim1}
\end{table}
\begin{table}[!ht]
	\centering
	\caption{{\tt rmse} and {\tt mnorm} under the covariance  structure {\tt C2}.}
	{\small
		\begin{tabular}{c|| l| l| rrr}
			\hline\hline
			 &  & Measure  & Proposed & SSSL & SampCov \\ \hline
			$n = 100$ & $\mu = 0.02 $ &{\tt rmse} & {\bf 0.034 (0.008)}  & 0.038 (0.007) & 0.109 (0.007) \\
			$(p = 50)$ &            & {\tt mnorm} & {\bf 0.950 (0.407)} & 0.994 (0.405)  & 1.137 (0.409) \\ 
			\cline{2-6}
			& $\mu = 0.1$ &{\tt rmse} & {\bf 0.063 (0.004)}  & 0.065 (0.005)  & 0.127 (0.007) \\
			&            & {\tt mnorm} & {\bf 0.985 (0.409)} & 0.998 (0.418)  & 1.176 (0.414) \\ 
			\cline{2-6}
			& $\mu = 0.5$ &{\tt rmse} & 0.287 (0.002)  & {\bf 0.281 (0.004)}  & 0.271 (0.014) \\
			&            & {\tt mnorm} & {\bf 1.371(0.560)} & 1.470 (0.530)  & 1.609 (0.560) \\ 
			\cline{2-6}
			& $\mu = 1$ &{\tt rmse} & 0.561 (0.003)  & {\bf 0.544 (0.006)}  & 0.467 (0.023)\\
			&            & {\tt mnorm} & {\bf 1.789 (0.397)} & 1.800 (0.370) & 2.117 (0.403) \\ 
			\hline\hline
			$n = 50$ & $\mu = 0.02 $ &{\tt rmse} & {\bf 0.043 (0.009)}  & 0.049 (0.009)  & 0.153 (0.010) \\
			$(p = 50)$&            & {\tt mnorm} & {\bf 1.238 (0.460)}  & 1.265 (0.472) & 1.526 (0.474) \\ 
			\cline{2-6}
			& $\mu = 0.1$ &{\tt rmse} & {\bf 0.071 (0.006)}  & 0.075 (0.007)  & 0.179 (0.010) \\
			&            & {\tt mnorm} & {\bf 1.263 (0.487)} & 1.306 (0.490)  & 1.583 (0.509) \\ 
			\cline{2-6}
			& $\mu = 0.5$ &{\tt rmse} & 0.292 (0.003)  & {\bf 0.285 (0.003)}  & 0.382 (0.019) \\
			&            & {\tt mnorm} & {\bf 1.564 (0.397)} & 1.580 (0.361)  & 2.056 (0.461) \\ 
			\cline{2-6}
			& $\mu = 1$ &{\tt rmse} & 0.582 (0.003)  & {\bf 0.558 (0.006)}  & 0.663 (0.034)\\
			&            & {\tt mnorm} & 2.278 (0.519) & {\bf 2.203 (0.534)}  & 2.991 (0.530) \\ 
			\hline\hline
		\end{tabular}
	}
	\label{tb:sim2}
\end{table}
\begin{figure}[!ht]
	\centering
	\includegraphics[width=11cm,height=11cm]{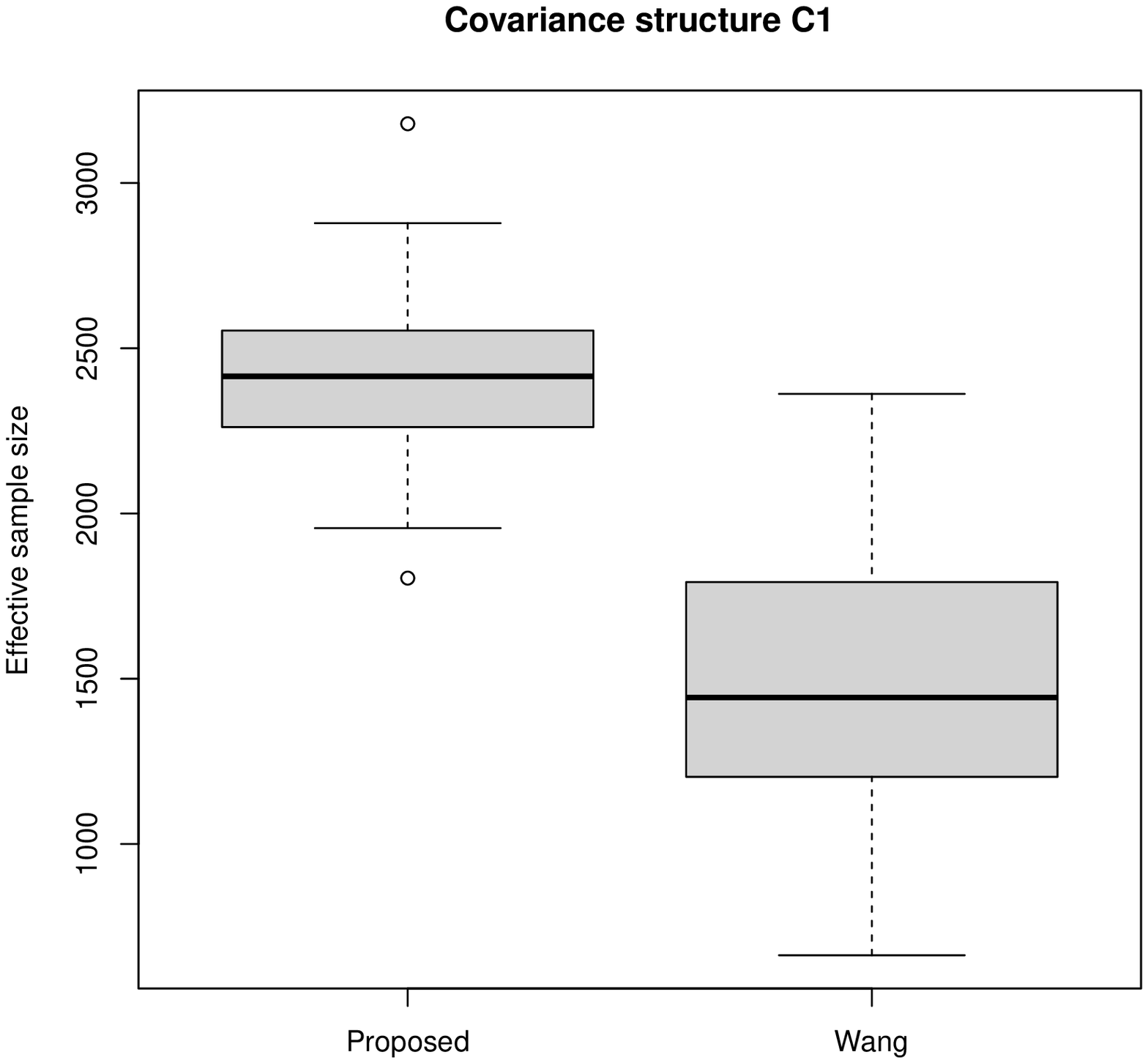}
	\caption{Effective sample size of the posterior samples of the shrinkage prior and the SSSL in case {\tt C1}.}
	\label{fig:ess1}
\end{figure}

\subsection{Real data application}
In this section, we consider two datasets to assess the performance of the shrinkage prior  for linear discriminant analysis (LDA) classification  \citep{anderson03}. The first is a colon cancer data, described in \cite{alon1999}, \cite{fisher11} and \cite{touloumis15}. The data set can be obtained from \url{http://genomics-pubs.princeton.edu/oncology/affydata}. 
It contains expression level measurements of  2000 genes on 40 normal and 22 colon tumor tissues. 

As the second example, we consider the leukemia data \citep{golub99,zhu04,guo07} which consists of 7128 gene expression measurements on 72 leukemia patients, 47 ``ALL" and 25 ``AML". This data set is available at \url{http://web.stanford.edu/~hastie}.

Following \cite{rothman08}, we select $p$ most significant genes with the two-sample $t$-statistic for $p = 50$ (colon data) and $p = 71$ (leukemia data) and then apply LDA to the datasets for classifying each observation in each data set into two groups. The LDA rule for an observation $X$ is given as
$$
\delta_j(X) = \argmax_j\left\{ X^\top \hat{\Sigma}^{-1}\hat{\boldsymbol{\mu}}_j - \frac{1}{2}\hat{\boldsymbol{\mu}}_j^{\top} \hat{\Sigma}^{-1}\hat{\boldsymbol{\mu}}_j + \log \hat{\omega}_j \right\}, ~~ j = 1, 2,
$$
where $\hat{\omega}_j$ is the proportion of class $j$ and $\hat{\boldsymbol{\mu}}_j$ is the sample mean for class $j$.

Table \ref{tb:realdata} shows leave-one-out cross validation (LOOCV) error rates for misclassified observations. From the result, we can see that our estimator outperforms two competitors, SampCov and SSSL for both data sets.
\begin{table}[!ht]
	\centering
	\caption{Classification error for colon and leukemia data.}
	{\small
		\begin{tabular}{l|| rrr}
			\hline\hline
			Dataset  & Proposed & SSSL &  SampCov \\ \hline
			Colon & 0.097  & 0.113 & 0.333 \\
			Leukemia & 0.014 & 0.028 & 0.486 \\ 
			\hline\hline
		\end{tabular}
	}
	\label{tb:realdata}
\end{table}

\section{Discussion}\label{sec:disc}
In this paper, we propose a theoretically supported shrinkage prior for sparse covariance matrices.
We prove that the proposed shrinkage prior achieves the minimax posterior convergence rate under the Frobenius norm in high-dimensional settings. The shrinkage prior performs better than or comparable to the SSSL method in the simulation studies, while is computationally more efficient than the SSSL. 
In our simulation study, the proposed shrinkage prior  is 4 times faster than the SSSL in terms of the computation time per ESS.
Two real data examples, colon and leukemia data, show the benefit of the LDA classification based on the proposed Bayesian method.

\section*{Acknowledgements}
This work was supported by the National Research Foundation of Korea (NRF) grant funded by the Korea government(MSIT) (No. 2020R1A4A1018207).
Seongil Jo was supported by Basic Science Research Program through the National Research Foundation of Korea (NRF) funded by the Ministry of Education (NRF-2017R1D1A3B03035235).

\newpage

\section*{Appendix A: Proof of Theorem \ref{thm:conv_rate}}
Let
\bea
K( f_{\sg_0}, f_\Sigma )  & := & \int f_{\sg_0} (x) \log \frac{ f_{\sg_0} (x) }{ f_{\Sigma} (x) } dx ,  \\
V(f_{\sg_0} , f_\Sigma )  & := & \int f_{\sg_0}  (x) \left( \log \frac{ f_{\sg_0} (x) }{ f_{\Sigma} (x) } \right)^2 dx ,
\eea 
where $f_\sg$ is the probability density function of $N_p(0, \sg)$ based on $n$ random samples $X_1,\ldots, X_n$. 
For a given $\epsilon > 0$, let 
\bea
B_\epsilon &:=& \Big\{ f_{\Sigma}: \,\, \Sigma \in \calC_p, \,\,  K( f_{\sg_0}, f_\Sigma )    < \epsilon^2, \,\,  V(f_{\sg_0} , f_\Sigma )  < \epsilon^2 \Big\}. 
\eea
To prove Theorem \ref{thm:conv_rate}, we apply Theorem 2.1 in \cite{ghosal2000convergence}.

\begin{lemma}[A version of Theorem 2.1 in \cite{ghosal2000convergence}]\label{lem:2000}
	Let $d$ be the Hellinger metric, and let $\calP = \{ f_\sg : \sg \in \calC_p \}$.
	Consider a sieve $\calP_n \subset \calP$ and  a sequence $\epsilon_n$ with $\epsilon_n\to 0$ and $n \epsilon_n^2\to\infty$.
	If, for some constants $C_1$ and $C_2>0$,
	\bean
	\log D(\epsilon_n, \calP_n , d ) &\le& C_1n \epsilon_n^2  \label{packing} \\
	\pi(\calP_n^c) &\le& \exp \big\{ - (C_2+4)n\epsilon_n^2  \big\}   , \label{prior}   \\
	\pi (  B_{\epsilon_n}  )  &\ge& \exp (- C_2 n\epsilon_n^2) ,    \label{KL}
	\eean
	then for sufficiently large $M>0$, we have
	\bea
	\pi \big(  d( f_{\sg_0}, f_\sg  ) > M \epsilon_n \mid \bfX_n \big) 
	&\lra& 0  
	\eea
	as $n\to\infty$ in $\bbP_{\sg_0}$-probability,  where $D(\epsilon, \calP_n , d )$ is the $\epsilon$-packing number of $\calP_n$ with respect to the distance $d$. 
\end{lemma}

Based on the above lemma, it suffices to show that conditions \eqref{packing}-\eqref{KL} hold under the assumptions in Theorem \ref{thm:conv_rate}, using 
\bea 
\epsilon_n &:=& \Big( \frac{(p + s_0)\log p}{n} \Big)^{1/2} .
\eea
For the rest, $\pi$ denotes the proposed shrinkage prior \eqref{sprior}.
With a slightly abuse of notation, $\pi$ also denotes the prior for $f_{\Sigma}$ induced by the shrinkage prior \eqref{sprior}.

\subsection*{A1. The upper bound of the packing number}

We define 
\bea
\calP_n  &=& \Big\{   f_\sg: |s(\sg, \delta_n)| \le s_n, \,\,  \tau^{-1} \le \lambda_{\min}(\sg)\le \lambda_{\max}(\sg)\le \tau ,   \,\,  \|\sg\|_{\max} \le L_n    \Big\}, \\
\mathcal{U}(\delta_n, s_n, L_n,\tau) &=& \Big\{  \sg\in \calC_p: |s(\sg, \delta_n)| \le s_n, \,\,  \tau^{-1} \le \lambda_{\min}(\sg)\le \lambda_{\max}(\sg)\le \tau ,   \,\,  \|\sg\|_{\max} \le L_n   \Big\}
\eea
for some positive constants $\delta_n, s_n,L_n$ and $\tau$,  where $\|\Sigma \|_{\max} = \max_{ij} |\sigma_{ij}|$ for $\Sigma = (\sigma_{ij})$.  
Here $\tau$ is a fixed large constant such that $\tau_0 < \tau$.
$\delta_n, s_n$ and $L_n$ are specified in the following theorem.
It gives the upper bound of the packing number in \eqref{packing}.

\begin{theorem}[The upper bound of the packing number]\label{thm:packing}
	If $\tau^4 \le p$, $p \asymp n^{\beta}$ for some $0<\beta<1$,
	$s_n = c_1 n\epsilon_n^2 / \log p$, $L_n = c_2 n\epsilon_n^2$ and $\delta_n = \epsilon_n/\tau^3$  for some constants $c_1>1$ and $c_2>0$,
	we have 
	\bea
	\log D(\epsilon_n, \calP_n , d ) &\le& (12 + 1/\beta )   c_1  n\epsilon_n^2  .
	\eea
\end{theorem}

\subsection*{A2. The upper bound of the prior mass on $\calP_n^c$}


\begin{lemma}\label{lower.bound}  
	If $a=b=1/2$, $\tau_1^2 \asymp 1/( n p^4 \tau^2 )$ and $\tau>3$, we have 
	\bea
	\pi^u( \Sigma \in \mathcal{U}(\tau) ) >  \Big\{   \frac{\lambda \tau}{8} \exp \big( - \frac{\lambda \tau}{4} - \frac{C}{\sqrt{n} } \big)  \Big\}^p  
	\eea
	for some constant $C>0$.
\end{lemma}

The following theorem gives the upper bound of the prior mass on $\calP_n^c$ in \eqref{prior}.

\begin{theorem}[The upper bound of the prior mass]\label{thm:sieve}
	If $\delta_n =  \epsilon_n/\tau^3$,
	$3< \tau \le  (\log p) /  \lambda $, 
	$a=b=1/2$, $\tau_1^2 \asymp 1/( n p^4 \tau^2 )$,
	$\tau^4  \ll (p+s_0)^2 \log p$, 
	we have
	\bea
	\pi(\calP_n^c) &\le& \exp   \big\{  - ( c_1- 1) n \epsilon_n^2/3   \big\}   .
	\eea
\end{theorem}

\subsection*{A3. The lower bound for $\pi (  B_{\epsilon_n}  )$}

\begin{lemma}\label{KV.ineq} 
	If $\sg_0 \in \calU(s_0, \tau_0)$ and $\sg \in \calU(\tau)$, we have
	\begin{enumerate}
		\item[(i)] $K(f_{\sg_0}, f_{\Sigma} ) \leq \tau^4 \tau_0^2 \|\Sigma - \sg_0 \|_F^2$;
		\item[(ii)] $V(f_{\sg_0}, f_{\Sigma} ) \leq \frac{3}{2} \tau^4 \tau_0^2  \|\Sigma - \sg_0 \|_F^2$. 
	\end{enumerate}
\end{lemma}

\begin{lemma}\label{lem:horseshoe_tail}
	If $a=b=1/2$ and $\tau_1^2 \asymp 1/( n p^4 \tau^2 )$,
	then
	\bea
	\pi_{ij}^u(x) &\ge& \sqrt{\frac{1}{2\pi^3}} \frac{\tau_1}{ x^2 }  ,
	\eea 
	for any $x >1$, where $\pi_{ij}^u(\sigma_{ij})$ is the unconstrained marginal prior density of $\sigma_{ij}$.  
\end{lemma}

The following theorem gives the lower bound for $\pi (  B_{\epsilon_n}  )$ in  \eqref{KL}.

\begin{theorem}[The lower bound for $\pi (  B_{\epsilon_n}  )$]\label{thm:KL}
	If $\sg_0 \in \calU(s_0,\tau_0)$ with $\tau_0 < \tau$, 
	$p \asymp n^{\beta}$ for some $0<\beta<1$,
	$\tau^4 \le p$,
	$\tau^2 \tau_0^2 \le  s_0  \log p $,
	$n \ge \max\big\{ 1   / \tau_0^4  , s_0  / (1-\tau_0/\tau)^2  \big\} \log p /\tau^4 $, 
	$p^{-1} <\lambda  < \log p /\tau_0$, 
	$a=b=1/2$ and $\tau_1^2  \asymp 1 / (n p^4 \tau^2)$,
	we have
	\bea
	\pi( B_{\epsilon_n} )  &\ge& \exp \Big\{ - \big( 5 + \frac{1}{\beta} \big)  n\epsilon_n^2  \Big\}  .
	\eea
\end{theorem}

\begin{proof}[Proof of Theorem \ref{thm:conv_rate}]
	The proof follows from Theorems \ref{thm:packing}, \ref{thm:sieve}, \ref{thm:KL} and Lemma \ref{lem:2000}, with $c_1 = 28 + 3/\beta$.
	Note that $\lambda > p^{-1}$ and $\tau^4 \le \min \{ p , s_0  \log p , (\log p)^4 / \lambda^4 \}$ hold for all sufficiently large $n$ because we assume $\tau = O(1)$ and $\lambda=O(1)$. 
	\hfill $\blacksquare$ 
\end{proof}

\section*{Appendix B: Proof of auxiliary results}
\begin{lemma}\label{lem:fro_ineq}
	For any $p\times p$ matrices $A$ and $B$, we have
	\bea
	\|A B \|_F &\le& \| A\| \, \|B\|_F .
	\eea
\end{lemma}
\begin{proof}
	Let $b_j$ be the $j$th column of $B$. 
	Then,
	\bea
	\|AB \|_F^2 &=& \|   ( Ab_1 , \ldots, Ab_p ) \|_F^2 \\
	&=& \sum_{j=1}^p \|A b_j\|_2^2   \\
	&\le& \|A\|^2 \sum_{j=1}^p \|b_j\|_2^2 \\
	&=& \|A\|^2\,  \|B\|_F^2  .  \quad \blacksquare
	\eea
\end{proof}

\begin{proof}[Proof of Theorem \ref{thm:packing}]
	Note that the $\epsilon_n$-packing number $D(\epsilon_n, \calP_n , d )$ is the maximal number of points in $\calP_n$ such that the distance between each pair is greater than or equal to $\epsilon_n$.
	By Lemma A.1  in \cite{banerjee2015bayesian}, 
	\bea
	d(f_{\sg_1} , f_{\sg_2} ) 
	&\le&  C\tau \| \Omega_1 - \Omega_2\|_F 
	\eea
	for some constant $C>0$ and any $\sg_1, \sg_2 \in \mathcal{U}(\delta_n, s_n, L_n,\tau)$, where $\Omega_i = \sg_i^{-1}$ for $i=1,2$.
	Further note that for any $\sg_1, \sg_2 \in \mathcal{U}(\delta_n, s_n, L_n,\tau)$,
	\bea
	\| \Omega_1 - \Omega_2\|_F  &\le& \|\Omega_1\| \, \|\Omega_2\| \, \| \sg_1 - \sg_2\|_F \\
	&\le& \tau^2  \| \sg_1 - \sg_2\|_F,
	\eea
	which gives
	\bea
	d(f_{\sg_1} , f_{\sg_2} ) 
	&\le& C\tau^3   \| \sg_1 - \sg_2\|_F  .
	\eea
	By the definition of the $\epsilon_n$-packing number  it implies that 
	\bea
	\log D(  \epsilon_n, \calP_n , d ) 
	&\le& \log  D(  \epsilon_n /(C\tau^3) ,  \mathcal{U}(\delta_n, s_n, L_n,\tau) , \|\cdot\|_F )  \\
	&\le& \log \Big\{  \Big(  \frac{C L_n\tau^3}{ \epsilon_n} \Big)^{p} \sum_{j=1}^{s_n} \Big(  \frac{2C L_n\tau^3}{\epsilon_n} \Big)^{j} \binom{\binom{p}{2} }{j}  \Big\} \\
	&\le&  p \log  \Big(  \frac{C L_n\tau^3}{ \epsilon_n} \Big) +   \log \Big\{  \sum_{j=1}^{s_n} \Big(  \frac{2C L_n\tau^3 }{ \epsilon_n} \Big)^{j}  \Big( \frac{p^2}{2} \Big)^j \Big\}        \\
	&\le& p \log  \Big(  \frac{CL_n\tau^3}{ \epsilon_n} \Big) +   s_n  \log  \Big(  \frac{2C L_n\tau^3 p^2 }{ \epsilon_n} \Big)       \\
	&\le& (p+ s_n) \log (2C L_n) +(p+ s_n) \log \tau^3 + (p+ s_n) \log (1/\epsilon_n ) + 2s_n \log p  \\
	&\le& 3 s_n \log p   + \frac{3}{4} (p+ s_n) \log p  +  \frac{1}{2} (p+ s_n) \log \frac{n}{(p+s_0)\log p} + 2 s_n \log p     \\
	&\le&  3 s_n \log p   + \frac{3}{4} (p+ s_n) \log p  +  \frac{1}{2\beta} (p+ s_n) \log p  + 2 s_n \log p \\
	&\le&  \{6 + 1/(2\beta) \} ( s_n /c_1 + s_n) \log p \,\,=\,\, (12 + 1/\beta )   c_1  n\epsilon_n^2
	\eea
	for all sufficiently large $n$, because $c_1>1$, $\tau^4 \le p$, $p \asymp n^{\beta}$,  $\delta_n = \epsilon_n / \tau^3$, $s_n = c_1 n\epsilon_n^2 / \log p$ and $L_n = c_2 n\epsilon_n^2$.   
	\hfill $\blacksquare$
\end{proof}

\begin{proof}[Proof of Lemma \ref{lower.bound}]
	By Gershgorin circle theorem,  every eigenvalue of $\sg$ lies within as least one of $[ \sigma_{jj}- \sum_{k \neq j} | \sigma_{kj} | , \sigma_{jj}+\sum_{k \neq j} | \sigma_{kj} |  ]$ for $j=1,\ldots,p$ \citep{brualdi1994regions}.
	Thus, it suffices to show 
	\bea 
	\pi^u(  \Sigma \in \mathcal{U}(\tau) ) &\ge& \pi^u \Big( \min_j\big(\sigma_{jj} - \sum_{k \neq j} | \sigma_{kj} |\big) > 0,  \,\, \tau^{-1}\le \lambda_{\min}(\sg)\le \lambda_{\max}(\sg)\le \tau  \Big)  .
	\eea
	On the event $\min_j\big(\sigma_{jj} - \sum_{k \neq j} | \sigma_{kj} |\big) > 0$, we have
	\bea
	\lambda_{\max}(\sg) 
	&\le& \|\sg\|_1 \\
	&=& \max_j \big(  \sigma_{jj} + \sum_{k \neq j} |\sigma_{kj}|  \big)  \\
	&\le& \max_j 2\sigma_{jj} 
	\eea
	and
	\bea
	\lambda_{\min}(\sg)
	&\ge& \min_j \big( \sigma_{jj} - \sum_{k \neq j} |\sigma_{kj}|   \big) .
	\eea
	Therefore,
	\bea
	&& \pi^u(  \Sigma \in \mathcal{U}(\tau) ) \\
	&\ge& \pi^u \Big(   \tau^{-1} \le  \min_j \big( \sigma_{jj} - \sum_{k \neq j} |\sigma_{kj}|   \big) \le 2\max_j \sigma_{jj} \le \tau     \Big) \\
	&=& \pi^u \Big(   \tau^{-1} \le  \min_j \big( \sigma_{jj} - \sum_{k \neq j} |\sigma_{kj}|   \big) \le 2\max_j \sigma_{jj} \le \tau   \,\,\big|\,\,   \max_{k\neq j}|\sigma_{kj}| < (\tau p)^{-1}   \Big) 
	\pi^u( \max_{k\neq j}|\sigma_{kj}| < (\tau p)^{-1} )  \\
	&\ge& \pi^u \Big(   \tau^{-1} \le  \min_j \big( \sigma_{jj} - \tau^{-1}\big) \le 2\max_j \sigma_{jj} \le \tau   \,\,\big|\,\,   \max_{k\neq j}|\sigma_{kj}| < (\tau p)^{-1}   \Big) 
	\pi^u( \max_{k\neq j}|\sigma_{kj}| < (\tau p)^{-1}  )  \\
	&=& \pi^u \Big(   \tau^{-1} \le  \min_j \big( \sigma_{jj} - \tau^{-1}\big) \le 2\max_j \sigma_{jj} \le \tau     \Big) 
	\pi^u( \max_{k\neq j}|\sigma_{kj}| < (\tau p)^{-1}  ) .
	\eea
	Note that
	\bea
	\pi^u \Big(   \tau^{-1} \le  \min_j \big( \sigma_{jj} - \tau^{-1}\big) \le 2\max_j \sigma_{jj} \le \tau     \Big) 
	&\ge&  \pi^u \Big(   2\tau^{-1} \le   \sigma_{jj}  \le \tau /2  ,\,\, \forall j  \Big) \\
	&=& \prod_{j=1}^p \pi^u \Big(   2\tau^{-1}  \le   \sigma_{jj}  \le \tau /2    \Big)  \\
	&\ge&  \Big\{   \big( \frac{\tau}{2} - 2\tau^{-1}    \big) \frac{\lambda}{2} \exp \big( - \frac{\lambda \tau}{4} \big)  \Big\}^p \\
	&\ge&  \Big\{   \frac{\lambda \tau}{8} \exp \big( - \frac{\lambda \tau}{4} \big)  \Big\}^p   
	\eea
	because $\tau>3$.
	Furthermore, by Theorem 1 in \cite{carvalho2010horseshoe} and the change of variables,	
	\bean\label{tail_prob}
	\pi^u(\sigma_{kj} ) &\le&  \frac{1}{\tau_1 \sqrt{2\pi^3}} \log \Big( 1+ \frac{2\tau_1^2 }{\sigma_{kj}^2} \Big)
	\eean
	for any $\sigma_{kj}\neq 0$, which implies 
	\bea
	\pi^u(  |\sigma_{kj}| \ge (\tau p)^{-1}  ) 
	&\le& \frac{1}{\tau_1} \sqrt{\frac{2}{\pi^3}} \int_{(\tau p)^{-1}}^\infty \log  \Big( 1+ \frac{2\tau_1^2}{x^2} \Big) dx \\
	&\le&  \sqrt{\frac{2}{\pi^3}} \int_{(\tau p)^{-1}}^\infty  \frac{2\tau_1}{x^2} dx \\
	&=& \frac{2\sqrt{2}}{\sqrt{\pi^3}}  \tau_1 \tau p  .
	\eea
	Thus, we have
	\bea
	\pi^u( \max_{k\neq j} |\sigma_{kj}| < (\tau p)^{-1}  ) 
	&=& \prod_{k\neq j} \Big\{   1 - \pi^u(  |\sigma_{kj}| \ge (\tau p)^{-1}  ) \Big\}  \\
	&\ge&  \Big(  1 -  \frac{2\sqrt{2}}{\sqrt{\pi^3}}  \tau_1 \tau p   \Big)^{p^2}  \\
	&\ge& \exp \Big(  - \frac{4\sqrt{2}}{\sqrt{\pi^3}}  \tau_1 \tau p^3   \Big)  \\
	&=& \exp \Big( - C \frac{p}{\sqrt{n}}  \Big) 
	\eea
	for some constant $C>0$, because $\tau_1^2  \asymp 1/( n p^4 \tau^2 )$. 
	\hfill $\blacksquare$
\end{proof}

\begin{proof}[Proof of Theorem \ref{thm:sieve}]
	Note that 
	\bea
	\pi(\calP_n^c)
	&\le&  \pi ( |s(\sg, \delta_n)| > s_n )  + \pi ( \|\sg\|_{\max} > L_n )    .
	\eea
	First, we focus on the upper bound for $\pi ( |s(\sg, \delta_n)| > s_n ) $.
	For any $1\le k \neq j \le p$, by applying inequality \eqref{tail_prob},  we have 
	\bea
	\nu_n \,\,\equiv\,\,  \pi^u ( |\sigma_{kj}| > \delta_n )
	&\le&  \frac{2\sqrt{2} }{ \sqrt{\pi^3} } \tau_1 \delta_n^{-1}  \\
	&\le& \frac{C}{\sqrt{n} \tau p^{2} }   \frac{\tau^3 \sqrt{n} }{\sqrt{(p+s_0)\log p} }    \\
	&=&    \frac{C \tau^2 }{ \sqrt{p^4 (p+s_0)\log p} }     
	\eea
	for some constant $C>0$, because $\tau_1^2 \asymp 1/(n p^4 \tau^2)$.
	Thus,
	\bea
	\pi^u ( |s(\sg, \delta_n)| > s_n ) 
	&=& \bbP \Big(  B\Big( \binom{p}{2} , \nu_n \Big)  > s_n    \Big)   \\
	&\le& 1 - \Phi  \Big(  \sqrt{2 \binom{p}{2}  H \Big(\nu_n ,  s_n/ \binom{p}{2}  \Big) }  \Big)
	\eea
	by Lemma A.3 of \cite{song2018nearly}, provided $s_n  > \binom{p}{2}  \nu_n$, where $\Phi$ is the cdf of $N(0,1)$ and
	\bea
	H( \nu, k/n ) &=& (k/n) \log \{k/ (n\nu)\} + (1-k/n) \log \{ (1-k/n) / (1-\nu) \}.
	\eea
	Note that the condition $s_n  > \binom{p}{2}  \nu_n$ is met because we assume that $\tau^4  \ll (p+s_0)^2 \log p$. 
	
	Note that
	\bea
	1 - \Phi  \Big(  \sqrt{2 \binom{p}{2}  H \Big(\nu_n ,  s_n/ \binom{p}{2}  \Big) }  \Big) &\le& \frac{ \exp \big[ - \binom{p}{2} H\{\nu_n, s_n/\binom{p}{2}  \} \big] }{ \sqrt{2\pi}\sqrt{2 \binom{p}{2} H\{ \nu_n,  s_n/ \binom{p}{2}  \} } },
	\eea
	where
	\bea
	\binom{p}{2}  H \Big( \nu_n, s_n/\binom{p}{2}  \Big)  &=& s_n \log \Big( \frac{s_n}{\binom{p}{2} \nu_n} \Big) + \Big\{ \binom{p}{2}  - s_n\Big\} \log \Big( \frac{\binom{p}{2}  - s_n}{\binom{p}{2}  - \binom{p}{2} \nu_n} \Big).
	\eea
	We have
	\bea
	s_n \log \Big( \frac{s_n}{\binom{p}{2} \nu_n} \Big)
	&\ge&  s_n \log \Big( c_1 C  \sqrt{\frac{ (p+s_0)^3 \log p }{\tau^4}}   \Big) \\
	&\ge& s_n \log \Big(     \sqrt{p+s_0}   \Big)  \\
	&\ge&  s_n  (\log p) /2   \,\,=\,\,   c_1  n \epsilon_n^2 /2
	\eea
	for some constant $C>0$ because $\tau^4  \ll (p+s_0)^2 \log p$ and $c_1 > 1$, and
	\bea
	\Big\{ \binom{p}{2}  - s_n\Big\} \log \Big( \frac{\binom{p}{2}  - s_n}{\binom{p}{2}  - \binom{p}{2} \nu_n} \Big)  
	&=& \Big\{ \binom{p}{2}  - s_n\Big\} \log \Big( 1 - \frac{s_n - \binom{p}{2}   \nu_n }{\binom{p}{2}( 1 -  \nu_n)} \Big)   \\
	&\ge& -\frac{1}{2}\Big\{ \binom{p}{2}  - s_n\Big\}  \frac{s_n - \binom{p}{2}   \nu_n }{\binom{p}{2}( 1 -  \nu_n)}  \\
	&\ge& - \frac{1}{2} \Big\{ 1 - s_n /\binom{p}{2} \Big\}  \frac{ s_n   }{   1 -  \nu_n }  \\
	&\gtrsim&  -s_n \Big\{ 1 - c_1 (p+s_0) /p^2 \Big\} \,\,\gtrsim\,\, -  \frac{c_1 n \epsilon_n^2}{\log p} 
	\eea
	for all sufficiently large $n$.
	Therefore, due to Lemma \ref{lower.bound} and $\lambda \tau \le  \log p $,
	\bea
	\pi ( |s(\sg, \delta_n)| > s_n ) 
	&\le& \pi^u ( |s(\sg, \delta_n)| > s_n )  /\pi^u( \sg \in \calU(\tau) ) \\
	&\le& \exp \big(  -  c_1 n \epsilon_n^2 /3  \big) /\pi^u( \sg \in \calU(\tau) ) \\
	&\le& \Big\{ \frac{8}{\lambda \tau} \exp \big( \frac{\lambda \tau}{4} + \frac{C }{\sqrt{n} } \big)  \Big\}^p \exp \big(  -  c_1 n \epsilon_n^2/3 \big) \\
	&\le& \exp   \big(  -  c_1 n \epsilon_n^2/3  + p \log p /3 \big) \\
	&\le& \exp   \big\{  - ( c_1- 1) n \epsilon_n^2/3   \big\}
	\eea
	for some constant $C>0$ and all sufficiently large $n$.
	
	Now we focus on the upper bound for the second term $\pi ( \|\sg\|_{\max} > L_n ) $.
	Since $\pi ( \sg:  \lambda_{\max}(\sg)\le \tau )  =1$ and $\|\sg\|_{\max}  \le \lambda_{\max}(\sg)$, it means that  $\pi ( \|\sg\|_{\max}  > L_n )  =0$ for all sufficiently large $n$, because $L_n\to\infty$ as $n\to\infty$.
	It completes the proof. 
	\hfill $\blacksquare$
\end{proof}

\begin{proof}[Proof of Lemma \ref{KV.ineq}]
	In the proof, we follow the calculation of \cite{banerjee2015bayesian}. 
	Let $d_i$'s be the  eigenvalues of ${\sg_0}^\half \Sigma^{-\half} {\sg_0}^\half$. 
	By Lemma A.1 in \cite{banerjee2015bayesian}, we obtain 
	\bea
	\sum_{i=1}^p (1-d_i)^2 &\le& \tau^2 \|{\sg_0}^{-1} - \Sigma^{-1} \|_F^2.
	\eea
	Also,
	\bea
	\|{\sg_0}^{-1} - \Sigma^{-1} \|_F &\le& \|\sg^{-1}  \| \, \|{\sg_0}^{-1}  \| \,  \|\Sigma - \sg_0 \|_F \\
	&\le& \tau \tau_0   \|\Sigma - \sg_0 \|_F  . 
	\eea
	Following the calculation of \cite{banerjee2015bayesian}, we obtain 
	\begin{eqnarray*}
		K(f_{\sg_0}, f_{\Sigma})  & = & -\half \sum_{i=1}^p \log d_i - \half \sum_{i=1}^p (1-d_i) \\
		& \leq & \sum_{i=1}^p (1-d_i)^2  \\
		& \leq & \tau^2 \|{\sg_0}^{-1} - \Sigma^{-1} \|_F^2 \\
		& \leq &   \tau^4 \tau_0^2 \|\Sigma - \sg_0 \|_F^2. 
	\end{eqnarray*}
	Similarly, we obtain
	\begin{eqnarray*} 
		V(f_{\sg_0}, f_{\Sigma} ) & = & \half \sum_{i=1}^p (1-d_i)^2 + K(f_{\sg_0}, f_{\Sigma} )^2 \\
		& \leq & \frac{3}{2}  \sum_{i=1}^p (1-d_i)^2 \\
		& \leq & \frac{3}{2}  \tau^2 \|{\sg_0}^{-1} - \Sigma^{-1} \|_F^2 \\
		& \leq & \frac{3}{2} \tau^4 \tau_0^2  \|\Sigma - \sg_0 \|_F^2.
	\end{eqnarray*}
	This completes the proof. \hfill $\blacksquare$
\end{proof}

\begin{proof}[Proof of Lemma \ref{lem:horseshoe_tail}]
	Because we have $a=b=1/2$,
	\bea
	\sigma_{ij}/\tau_1 \mid \rho_{ij} &\sim& N\Big(  0, \frac{\rho_{ij}}{1-\rho_{ij}} \Big) ,\\
	\rho_{ij} &\sim& Beta(a,b)
	\eea
	is equivalent to
	\bea
	\sigma_{ij}/\tau_1 \mid \lambda_{ij} &\sim& N \Big(  0, \lambda_{ij}^2 \Big)  , \\
	\lambda_{ij} &\sim& C^+ (0,1) ,
	\eea
	where $C^+ (0,s)$ denotes the standard half-Cauchy distribution on positive real with a scale parameter $s$.
	Then, by Theorem 1 in \cite{carvalho2010horseshoe} and the change of variables, 
	\bea
	\pi_{ij}^u ( x ) 
	&\ge& \frac{1}{2\tau_1} \sqrt{\frac{1}{2\pi^3}}  \log \Big( 1+ \frac{4\tau_1^2 }{x^2}  \Big)  \\
	&\ge& \frac{1}{4\tau_1} \sqrt{\frac{1}{2\pi^3}}  \frac{4\tau_1^2 }{x^2}   \\
	&\ge& \sqrt{\frac{1}{2\pi^3}} \frac{\tau_1}{x^2} 
	\eea
	for any $x>1$, because $\tau_1 \asymp 1/(n p^4 \tau^2)$.  \quad $\blacksquare$
\end{proof}

\begin{proof}[Proof of Theorem \ref{thm:KL}]
	By Lemma \ref{KV.ineq}, it suffices to show that 
	\bea
	\pi \Big( \|\sg- \sg_0\|_F^2 \le \frac{2}{3 \tau^4 \tau_0^2 }\epsilon_n^2  \Big)
	&\ge& \exp \big( - Cn\epsilon_n^2 \big)  .
	\eea
	Note that
	\bea
	&& \pi \Big( \|\sg- \sg_0\|_F^2 \le \frac{2}{3 \tau^4 \tau_0^2 }\frac{(p+s_0)\log p}{n}  \Big)\\
	&\ge& \pi \Big( \sum_{i\neq j}(\sigma_{ij}-\sigma_{ij}^*)^2 \le  \frac{2}{3 \tau^4 \tau_0^2 }\frac{s_0 \log p}{n}  , \,\, \sum_{j=1}^p(\sigma_{jj}-\sigma_{jj}^*)^2 \le  \frac{2}{3 \tau^4 \tau_0^2 }\frac{p \log p}{n}     \Big)   \\
	&\ge&  \pi \Big( \max_{i\neq j}(\sigma_{ij}-\sigma_{ij}^*)^2 \le  \frac{2}{3 \tau^4 \tau_0^2 }\frac{s_0 \log p}{p(p-1) n}  , \,\,  \max_{1\le j \le p}(\sigma_{jj}-\sigma_{jj}^*)^2 \le  \frac{2}{3 \tau^4 \tau_0^2 }\frac{\log p}{n}     \Big)   \\
	&\equiv& \pi ( A_{n, \sg_0} )  ,
	\eea
	where $\Sigma_0 = ( \sigma_{ij}^*)$.
	By Weyl's theorem,  if $\sg \in A_{n, \sg_0}$,
	\bea
	\lambda_{\min}(\sg) 
	&\ge& \lambda_{\min}(\sg_0) - \| \sg - \sg_0 \| \\
	&\ge& \lambda_{\min}(\sg_0) - \| \sg - \sg_0 \|_1  \\
	&\ge& \tau_0^{-1}  - p\sqrt{\frac{2}{3 \tau^4 \tau_0^2 }\frac{s_0 \log p}{p(p-1) n} }  - \sqrt{\frac{2}{3 \tau^4 \tau_0^2 }\frac{\log p}{n}   } \\
	&\ge& \tau^{-1}
	\eea
	and
	\bea
	\lambda_{\max}(\sg) 
	&\le& \lambda_{\max}(\sg_0) + \| \sg - \sg_0 \| \\ 
	&\le& \tau_0 + \|\sg - \sg_0 \|_1 \\
	&\le& \tau_0  + p\sqrt{\frac{2}{3 \tau^4 \tau_0^2 }\frac{s_0 \log p}{p(p-1) n} }  + \sqrt{\frac{2}{3 \tau^4 \tau_0^2 }\frac{\log p}{n}   }  \\
	&\le& \tau 
	\eea
	for all sufficiently large $n$, because $\sg_0 \in \calU(s_0,\tau_0)$, $\tau_0 < \tau$ and $\tau^4 (1-\tau_0/\tau)^2 n \ge s_0 \log p$.
	Thus, if $\sg \in A_{n,\sg_0}$, then we have $\sg \in \calU(\tau)$.
	Because 
	\bea
	{\pi}( \Sigma ) &=& \frac{\pi^u(\Sigma) I (\Sigma \in \mathcal{U}(\tau) ) }{ \pi^u( \Sigma \in \mathcal{U}(\tau)  ) } ,
	\eea
	we have 
	\bea
	\pi ( A_{n, \sg_0} )  &\ge& \pi^u ( A_{n, \sg_0} )  .
	\eea
	
	Note that 
	\bea
	\pi^u ( A_{n, \sg_0} ) 
	&=& \pi^u \Big( \max_{i\neq j}(\sigma_{ij}-\sigma_{ij}^*)^2 \le  \frac{2}{3 \tau^4 \tau_0^2 }\frac{s_0 \log p}{p(p-1) n}  \Big) \times \pi^u \Big(  \max_{1\le j \le p}(\sigma_{jj}-\sigma_{jj}^*)^2 \le  \frac{2}{3 \tau^4 \tau_0^2 }\frac{\log p}{n}     \Big)  \\
	&=& \prod_{i<j}\pi^u \Big( (\sigma_{ij}-\sigma_{ij}^*)^2 \le  \frac{2}{3 \tau^4 \tau_0^2 }\frac{s_0 \log p}{p(p-1) n}  \Big) \times \prod_{j=1}^p \pi^u \Big(  (\sigma_{jj}-\sigma_{jj}^*)^2 \le  \frac{2}{3 \tau^4 \tau_0^2 }\frac{\log p}{n}     \Big)   
	\eea
	and
	\bea
	\prod_{j=1}^p \pi^u \Big(  (\sigma_{jj}-\sigma_{jj}^*)^2 \le  \frac{2}{3 \tau^4 \tau_0^2 }\frac{\log p}{n}     \Big)   
	&\ge& \prod_{j=1}^p 2 \sqrt{\frac{2}{3 \tau^4 \tau_0^2 }\frac{\log p}{n} } \, \, \frac{\lambda}{2}  \exp \Big\{ - \frac{\lambda}{2} \Big(  \sigma_{jj}^* + \sqrt{\frac{2}{3 \tau^4 \tau_0^2 }\frac{\log p}{n} }  \Big)  \Big\} \\
	&\ge& \left[ \lambda\sqrt{\frac{2}{3 \tau^4 \tau_0^2 }\frac{\log p}{n} }  \exp \Big\{ - \frac{\lambda}{2} \Big(  \tau_0 + \sqrt{\frac{2}{3 \tau^4 \tau_0^2 }\frac{\log p}{n} }  \Big)  \Big\} \right]^p  \\
	&\ge& \exp \Big\{   -p \lambda \tau_0  - p \log \Big( \frac{\sqrt{3\tau^4\tau_0^2 n}}{\lambda \sqrt{2\log p}}  \Big)   \Big\}    \\
	&\ge& \exp \Big\{  - p\log p - p \log \Big( \frac{\tau^3 p^{1/(2\beta)} }{\lambda} \Big)  \Big\}   \\
	&\ge& \exp \Big\{  - p\log p - \big(2 + \frac{1}{2\beta} \big) p \log p  \Big\}  \\
	&=& \exp \Big\{   - \big(3 + \frac{1}{2\beta} \big) p \log p  \Big\} 
	\eea
	for all sufficiently large $n$,
	because $\log p  /( \tau^4 \tau_0^4  ) \le  n $, $\tau^4 \le p$ and $p^{-1} <\lambda  < \log p /\tau_0$.
	Furthermore,
	\bean
	&& \prod_{i<j}\pi^u \Big( (\sigma_{ij}-\sigma_{ij}^*)^2 \le  \frac{2}{3 \tau^4 \tau_0^2 }\frac{s_0 \log p}{p(p-1) n}  \Big) \nonumber\\
	&\ge& \prod_{(i,j)\in s(\sg_0) } \pi^u  \Big( (\sigma_{ij}-\sigma_{ij}^*)^2 \le  \frac{2}{3 \tau^4 \tau_0^2 }\frac{s_0 \log p}{p(p-1) n}  \Big) 
	\prod_{(i,j)\notin s(\sg_0), i<j } \pi^u  \Big( \sigma_{ij}^2 \le  \frac{2}{3 \tau^4 \tau_0^2 }\frac{s_0 \log p}{p(p-1) n}  \Big) . \,\,\,\,\label{lbound}
	\eean
	The second term in \eqref{lbound} is bound below by
	\bea
	\prod_{(i,j)\notin s(\sg_0), i<j } \left\{  1 -\pi^u  \Big( \sigma_{ij}^2 >  \frac{2}{3 \tau^4 \tau_0^2 }\frac{s_0 \log p}{p(p-1) n}  \Big)  \right\}  
	&\ge&   \left\{  1 -  2\tau_1 \sqrt{\frac{2}{\pi^3} } \sqrt{\frac{3 \tau^4 \tau_0^2 }{2} \frac{p(p-1) n}{s_0 \log p}}  \right\}^{p^2}   \\
	&\ge&  \exp \Big(  - 4\sqrt{\frac{3}{\pi^3}}  \tau_1 \tau^2 \tau_0  p^3 \sqrt{\frac{n}{s_0 \log p}}    \Big) \\
	&\ge& \exp \Big(  - C  \tau \tau_0  p \sqrt{\frac{1}{s_0 \log p}}   \Big)  \\
	&\ge& \exp (  -C p   )
	\eea
	for some constant $C>0$, because $\tau^2 \tau_0^2 \le  s_0  \log p $ and $\tau_1^2 \asymp 1/(n p^4  \tau^2 )$.
	Let $\pi_{ij}^u(\sigma_{ij}) = \int_0^1 \pi^u(\sigma_{ij}\mid \rho_{ij}) \pi^u(\rho_{ij}) d\rho_{ij}$ be the unconstrained marginal prior density of $\sigma_{ij}$.
	The first term in \eqref{lbound} is bounded below by
	\bea
	&& \prod_{(i,j)\in s(\sg_0) } \pi_{ij}^u  \Big( \sigma_{ij}^* + \sqrt{\frac{2}{3 \tau^4 \tau_0^2 }\frac{s_0 \log p}{p(p-1) n}}\,\,  \Big)  \,\,  2 \sqrt{\frac{2}{3 \tau^4 \tau_0^2 }\frac{s_0 \log p}{p(p-1) n}}  \\
	&\ge&  \left\{  2\pi_{ij}^u  ( 2\tau_0 ) \sqrt{ \frac{2s_0 \log p}{3 \tau^4 \tau_0^2p(p-1) n}}  \right\}^{s_0}  \\
	&\ge&  \exp \Big\{    s_0\log (\pi_{ij}^u  ( 2\tau_0 )  )  -  \frac{1}{2} s_0 \log \Big(  \frac{3 \tau^4 \tau_0^2 p^2 n}{2 s_0 \log p } \Big)  \Big\} \\
	&\ge&  \exp \Big\{    s_0\log (\pi_{ij}^u  ( 2\tau_0 )  )  -  \frac{1}{2} s_0 \log \Big(  \frac{3 \tau^2 p^2 n}{2 } \Big)  \Big\} \\
	&\ge&  \exp \Big\{    -\frac{1}{2}s_0 \log \big( \tau_0^4 \tau^2 n p^4 \big)  -  \frac{1}{2} s_0 \log \Big(  \frac{3 \tau^2 p^2 n}{2 } \Big)   \Big\} \\
	&\ge& \exp \Big\{   -\frac{1}{2}s_0 \log \Big(  \frac{3 \tau_0^4 \tau^4 p^6 n^2}{2 } \Big)      \Big\} \\
	&\ge& \exp \Big\{   -\frac{1}{2}s_0 \log \Big(  C p^{8 + 2/\beta} \Big)      \Big\} \\
	&\ge& \exp \Big\{   -  \big( 5 + \frac{1}{\beta} \big) s_0\log p \Big\}
	\eea
	for some constant $C>0$, because $p \asymp n^{\beta}$, $\log p  /( \tau^4 \tau_0^4  ) \le  n  $, $\pi_{ij}^u(2\tau_0) \ge  \tau_1 /(4 \sqrt{2 \pi^3} \tau_0^2 ) \asymp 1/(\tau_0^2 \tau \sqrt{n} p^{2})$ and $\tau^4 \le p$, by Lemma \ref{lem:horseshoe_tail}.
	Therefore, we have
	\bea
	\pi^u ( A_{n, \sg_0} ) 
	&\ge& \exp \Big\{    - \big( 3 + \frac{1}{2\beta} \big) p \log p  -Cp  -  \big( 5 + \frac{1}{\beta} \big) s_0\log p \Big\} \\
	&\ge& \exp \Big\{   - \big( 5 + \frac{1}{\beta} \big)  (p+s_0)\log p \Big\} \,\,=\,\, \exp \Big\{ - \big( 5 + \frac{1}{\beta} \big)  n\epsilon_n^2  \Big\} . \quad \blacksquare
	\eea 
\end{proof}

\section*{Appendix C: Proof of Theorem \ref{thm:lower}}

\begin{proof}
	We will first show that
	\bean\label{lower1}
	\inf_{\hat{\Sigma}} \sup_{\Sigma_0 \in  B_1 } \bbE_0 \| \hat{\sg} - \sg_0 \|_F^2  &\gtrsim& \frac{s_0 \log p}{n} 
	\eean
	for some $B_1 \subset \mathcal{U}(s_0, \tau_0) $ when $s_0 > 3p$, and show that
	\bean\label{lower2}
	\inf_{\hat{\Sigma}} \sup_{\Sigma_0 \in  B_2 } \bbE_0 \| \hat{\sg} - \sg_0 \|_F^2  &\gtrsim& \frac{p}{n} 
	\eean
	for some $B_2 \subset \mathcal{U}(s_0, \tau_0) $ when $s_0 \le 3p$.

	\paragraph{(i) Proof of \eqref{lower1}}
	Let $r = \lfloor p/2 \rfloor$ and $\epsilon_{np} = \nu \sqrt{\log p /n}$ with $\nu = \sqrt{\epsilon/4}$.
	For any $u \in \bbR^p$, let $A_m(u)$ be a $p\times p$ symmetric matrix whose the $m$th row and column are equal to $u$ and the rest of entries are zero.
	Define a parameter space 
	\bea
	B_1 &:=& \Big\{ \sg(\theta) :  \sg(\theta) = I_p + \epsilon_{np} \sum_{m=1}^r \gamma_m A_m (\lambda_m) , \,\, \theta=(\gamma,\lambda) \in \Theta    \Big\}, 
	\eea
	where  
	$\gamma=(\gamma_1,\ldots, \gamma_r) \in \Gamma = \{0,1\}^r$, $\lambda=(\lambda_1,\ldots, \lambda_r)^T \in \Lambda \subset \bbR^{r \times p}$ and $\Theta = \Gamma \times \Lambda$.
	Here, we let 
	\bea
	\Lambda &:=& \big\{ \lambda= (\lambda_1,\ldots, \lambda_r)^T : \lambda_m =(\lambda_{mi})\in \{0,1\}^p,  \,\,
	\|\lambda_m \|_0 = k ,  \,\, \sum_{i=1}^{p-r}\lambda_{mi}=0   \\
	&& \quad\quad\quad\quad\quad\text{ for any } m=1,\ldots,r, \text{ and satisfies } \max_{1\le i \le p} \sum_{m=1}^r \lambda_{mi} \le 2k    \big\}  ,
	\eea
	$k = \lceil c_{np}/2 \rceil -1$ and $c_{np} = \lceil s_0/p \rceil$.
	
	We will first show that $B_1 \subset \mathcal{U}(s_0, \tau_0) $.
	Note that $\|\sg(\theta)\| \le \|\sg(\theta)\|_1 \le 1 + 2k\epsilon_{np} \le 1 +  c_{np} \nu \sqrt{\log p /n} \le \tau_0$ for any $\tau_0>1$ and sufficiently large $n$  due to our assumption, $s_0^2 (\log p)^3 = O(p^2 n)$ . 
	Also note that $2k \epsilon_{np} \le c_{np} \nu \sqrt{\log p/n} \le (1+ s_0/p) \nu \sqrt{\log p/n} \le 1 -\tau_0^{-1}$ for any $\tau_0>1$ and sufficiently large $n$, which implies that $\sg(\theta) - \tau_0^{-1} I_p$ is diagonally dominant.
	Thus, we have $\lambda_{\min}(\sg(\theta)) \ge \tau_0^{-1}$.
	Because $|s(\sg(\theta))| \le 2k p \le s_0$, it holds that $B_1 \subset \mathcal{U}(s_0, \tau_0) $.

	For given $\theta\in\Theta$ and $a \in \{0,1\}$, let $\bbP_\theta$ and $\bar{\bbP}_{i,a}$ be the joint distribution of random samples $X_1,\ldots, X_n$ from $N_p(0, \sg(\theta) )$ 
	and $\{2^{r-1}|\Lambda| \}^{-1} \sum_{\theta\in \Theta_{i,a} } \bbP_\theta$, respectively, where $\Theta_{i,a} = \{ \theta\in \Theta: \gamma_i(\theta) =a \}$.
	For any two probability measures $\bbP$ and $\mathbb{Q}$, let $\|\bbP \wedge \mathbb{Q}\| = \int (p \wedge q)  d\mu$, where $p$ and $q$ are probability densities corresponding to $\bbP$ and $\mathbb{Q}$, respectively, with respect to a common dominating measure $\mu$.
	By applying Lemma 3 of \cite{cai2012optimal} with $s=2$, we have
	\bea
	\inf_{\hat{\sg}} \max_{\theta\in\Theta} 2^2 \bbE_\theta \| \hat{\sg} - \sg(\theta)\|_F^2 
	&\ge& \alpha \, \frac{r}{2} \, \min_{1\le i \le r} \| \bar{\bbP}_{i,0} \wedge \bar{\bbP}_{i,1}  \|,
	\eea 
	where $\bbE_\theta$ denotes the expectation with respect to $X_1,\ldots, X_n \overset{iid}{\sim} N_p(0, \sg(\theta))$ and
	$$\alpha = \min_{(\theta,\theta'):  H(\gamma(\theta), \gamma(\theta') ) \ge 1 } \|\sg(\theta) - \sg(\theta')\|_F^2/ H(\gamma(\theta), \gamma(\theta') )  .$$
	Here, $H( x, y ) = \sum_{j=1}^r |x_j- y_j|$ for any $x, y \in \{0,1\}^r$.
	By the definition of $\sg(\theta)$,  
	\bea
	\|\sg(\theta)-\sg(\theta')\|_F^2 &\ge& 2k \epsilon_{np}^2 H( \gamma(\theta), \gamma(\theta') )
	\eea
	for any $\theta, \theta' \in \Theta$, which implies 
	\bea
	\alpha  r	&\ge& 2 k  \epsilon_{np}^2 \, r \,\,\ge\,\, \nu^2 \Big( \frac{1}{2}  - \frac{p}{s_0} \Big) \frac{ s_0 \log p}{n} \,\,\asymp\,\,  \frac{ s_0 \log p}{n}
	\eea
	due to $s_0 > 3p$.
	Therefore, we complete the proof if we show that 
	\bean\label{p_12}
	\min_{1\le i \le r} \| \bar{\bbP}_{i,0} \wedge \bar{\bbP}_{i,1}  \| &\ge& c_1
	\eean
	for some constant $c_1 >0$.
	Note that, without loss of generality, it suffices to show that $\| \bar{\bbP}_{1,0} \wedge \bar{\bbP}_{1,1}  \| \ge c_1$.

	Let $\Lambda_1 = \{ \lambda_1(\theta) \in\bbR^p :  \theta\in \Theta \}$ and $\Lambda_{-1} = \{  \lambda_{-1}(\theta) \equiv (\lambda_2(\theta),\ldots, \lambda_r(\theta))^T  \in \bbR^{(r-1) \times p} :  \theta \in \Theta  \}$.
	For any $a\in\{0,1\}$, $b \in \{0,1\}^{r-1}$ and $c \in \Lambda_{-1}$, we define 
	\bea
	\bar{\bbP}_{(1,a,b,c)} &=& \frac{1}{|\Theta_{(1,a,b,c)}| }  \sum_{\theta\in \Theta_{(1,a,b,c)} } \bbP_\theta ,\\
	\Theta_{(1,a,b,c)}  &=& \big\{  \theta\in\Theta: \gamma_1(\theta)=a,  \gamma_{-1}(\theta)=b,   \lambda_{-1}(\theta) =c  \big\}  ,
	\eea
	where $\gamma_{-1}(\theta) \equiv (\gamma_{2}(\theta), \ldots, \gamma_{r}(\theta) ) \in \{0,1\}^{r-1}$ .
	Let $\bbE_{(\gamma_{-1}, \lambda_{-1})} f(\gamma_{-1}, \lambda_{-1})$ be the expectation of $f(\gamma_{-1}, \lambda_{-1})$ over $\Theta_{-1} = \{0,1\}^{r-1} \times \Lambda_{-1}$, i.e., 
	\bea
	\bbE_{(\gamma_{-1}, \lambda_{-1})} f(\gamma_{-1}, \lambda_{-1})
	&=& \frac{1}{2^{r-1} |\Lambda| }  \sum_{(b,c) \in \Theta_{-1} }  |\Theta_{(1,a,b,c)}| f(b, c)  ,
	\eea
	where the probability distribution of $(\gamma_{-1}, \lambda_{-1})$ is induced by the uniform distribution over $\Theta$.
	To show \eqref{p_12}, it suffices to prove that there exists $0<c_2<1$ such that 
	\bean\label{div_c2}
	\bbE_{(\gamma_{-1}, \lambda_{-1})} \Big\{  \int \Big(  \frac{d \bar{\bbP}_{(1,1,\gamma_{-1},\lambda_{-1})} }{d \bar{\bbP}_{(1,0,\gamma_{-1},\lambda_{-1})} }  \Big)^2  d \bar{\bbP}_{(1,0,\gamma_{-1},\lambda_{-1})}   -1 \Big\}  &\le& c_2^2  ,
	\eean
	by Lemma 8 (ii) of \cite{cai2012optimal}.
	
	Since $\bar{\bbP}_{(1,0,\gamma_{-1},\lambda_{-1})}$ assumes that $\gamma_1=0$, this is the distribution function of the $p$-dimensional normal distribution with a zero mean vector and a covariance matrix 
	\bea
	\sg_0 &=& \begin{pmatrix}
		1 & 0_{1\times (p-1)}  \\
		0_{(p-1)\times 1} & S_{(p-1)\times (p-1)}
	\end{pmatrix}  ,
	\eea
	where $S_{(p-1)\times (p-1)} = (s_{ij}) \in \bbR^{(p-1)\times (p-1)}$ is a symmetric matrix uniquely determined by  $(\gamma_{-1},\lambda_{-1})$: 
	\bea
	s_{ij} = \begin{cases}
		1  , \quad i=j, \\
		\epsilon_{np} , \quad \gamma_{i+1} = \lambda_{i+1}(j+1) = 1 ,  \\
		0 , \quad \text{ otherwise}.
	\end{cases}
	\eea
	Let 
	\bea
	\Lambda_1(c) &=& \big\{  a \in \bbR^p :  \lambda_1(\theta)=a , \lambda_{-1}(\theta) =c \,\, \text{ for some } \theta \in\Theta  \big\} 
	\eea
	be the set of all possible values of the first row, $\lambda_1(\theta)$, given the rest of the rows, $\lambda_{-1}(\theta) =c$.
	For a given $\lambda_{-1} \equiv \lambda_{-1}(\theta)  = (\lambda_2(\theta), \ldots, \lambda_r(\theta))^T \in \bbR^{(r-1)\times p}$, denote $n_{\lambda_{-1}}$ be the number of columns of $\lambda_{-1}$ whose sums are equal to $2k$, and let $p_{\lambda_{-1}} = r- n_{\lambda_{-1}}$.
	Then, by the definition of $\Theta$, $p_{\lambda_{-1}}$ is the number of entries in $\lambda_1$ which can be either $0$ or $1$.
	Note that $|\Lambda_1(\lambda_{-1}) | = \binom{p_{\lambda_{-1}}}{k}$ and $p_{\lambda_{-1}} \ge p/4 -1$ for any $\lambda_{-1}$.
	Then, $\bar{\bbP}_{(1,1,\gamma_{-1},\lambda_{-1})}$ is an average of $\binom{p_{\lambda_{-1}}}{k}$ normal distributions with covariance matrices of the form:
	\bean\label{cov_P1}
	\begin{pmatrix}
		1 & r^T    \\
		r & S_{(p-1)\times (p-1)}
	\end{pmatrix},
	\eean
	where $\|r\|_0 = k$ and nonzero entries of $r$ are equal to $\epsilon_{np}$.
	For a given $(\gamma_{-1}, \lambda_{-1})$, let $\sg_1$ and $\sg_2$ be covariance matrices of the form \eqref{cov_P1} with the first row  $\lambda_1 \in \Lambda_1(\lambda_{-1})$ and $\lambda_1' \in \Lambda_1(\lambda_{-1})$, respectively.
	Then, by the similar arguments in page 2411 of \cite{cai2012optimal},
	\bean
	&&  \bbE_{(\gamma_{-1}, \lambda_{-1})}  \Big\{   \int \Big(  \frac{d \bar{\bbP}_{(1,1,\gamma_{-1},\lambda_{-1})} }{d \bar{\bbP}_{(1,0,\gamma_{-1},\lambda_{-1})} }  \Big)^2  d \bar{\bbP}_{(1,0,\gamma_{-1},\lambda_{-1})}   -1   \Big\} \nonumber\\
	&=&  \bbE_{(\gamma_{-1}, \lambda_{-1})}  \left[ \bbE_{(\lambda_1,\lambda_1' )\mid \lambda_{-1} } \Big\{   \exp \big( \frac{n}{2} R_{\lambda_1,\lambda_1'}^{\gamma_{-1},\lambda_{-1} } \big)  -1\Big\}     \right]  \nonumber\\
	&=&  \bbE_{ (\lambda_1,\lambda_1' )   }  \left[ \bbE_{(\gamma_{-1}, \lambda_{-1})\mid (\lambda_1,\lambda_1' )} \Big\{   \exp \big( \frac{n}{2} R_{\lambda_1,\lambda_1'}^{\gamma_{-1},\lambda_{-1} } \big)  -1\Big\}     \right]  ,  \label{cond_exp}
	\eean
	where $\lambda_1,\lambda_1' \mid \lambda_{-1} \overset{iid}{\sim} Unif \{ \Lambda_1(\lambda_{-1}) \} $, $(\gamma_{-1}, \lambda_{-1})\mid (\lambda_1,\lambda_1' ) \sim Unif \{ \Theta_{-1}(\lambda_1, \lambda_1')  \}$,  
	\bea
	\Theta_{-1}(a_1,  a_2) &=&  \{0,1\}^{r-1} \times \big\{  c \in \Lambda_{-1} : \exists \theta_i\in\Theta, i=1,2 \text{ such that } \lambda_1(\theta_i)=a_i , \lambda_{-1}(\theta_i) = c  \big\}   
	\eea
	and
	\bean\label{Rlam}
	R_{\lambda_1,\lambda_1'}^{\gamma_{-1},\lambda_{-1} }  = - \log \det \{ I_p - \sg_0^{-2}(\sg_0- \sg_1)(\sg_0- \sg_2) \}  .
	\eean 
	By Lemma \ref{lem:Rdecomp}, \eqref{cond_exp} is bounded above by
	\bean
	&& \bbE_J \left[  \exp \big\{  -n \log (1- J \epsilon_{np}^2) \big\}   \bbE_{(\lambda_1,\lambda_1') \mid J} \Big\{    \bbE_{(\gamma_{-1}, \lambda_{-1}) \mid (\lambda_1,\lambda_1')} \exp \big(\frac{n}{2} R_{1, \lambda_1,\lambda_1'}^{\gamma_{-1}, \lambda_{-1} } \big)    \Big\}   -1 \right]  \nonumber\\
	&\le& \bbE_J  \Big[  \exp \Big\{   -n \log ( 1- J \epsilon_{np}^2)   \Big\} \, \frac{3}{2}   -1\Big]  ,  \label{Exp_J}
	\eean
	where $J $ is the number of overlapping nonzero entries between the first rows of $\sg_1$ and $\sg_2$, i.e., $J = \lambda_1^T \lambda_1'$.
	Note that for any $0\le j \le k$,
	\bea
	\bbE_J \big\{ I(J = j)  \mid \lambda_{-1}   \big\}  &=&  \frac{\binom{k}{j} \binom{p_{\lambda_{-1}}- k }{k - j}  }{\binom{p_{\lambda_{-1}} }{k} }  \\
	&=&  \Big\{  \frac{k!}{(k-j)!}  \Big\}^2   \frac{\{(p_{\lambda_{-1}} -k)! \}^2 }{p_{\lambda_{-1}}! (p_{\lambda_{-1}} - 2k + j)! } \frac{1}{j!}  \\
	&\le& \Big(  \frac{k^2}{p_{\lambda_{-1}}-k }  \Big)^j ,
	\eea
	because $\lambda_1,\lambda_1' \mid \lambda_{-1} \overset{iid}{\sim} Unif \{ \Lambda_1(\lambda_{-1}) \} $.
	Then, we have 
	\bea
	\bbE_J I(J= j)  &=& \bbE_{\lambda_{-1}} \Big[  \bbE_J \big\{  I(J=j)   \mid \lambda_{-1} \big\}  \Big]  \\
	&\le&  \bbE_{\lambda_{-1}} \Big\{      \Big(  \frac{k^2}{p_{\lambda_{-1}}-k }  \Big)^j  \Big\}    \\
	&\le&  \Big(  \frac{k^2}{ p/4 - 1 -k }  \Big)^j  
	\eea
	because $p_{\lambda_{-1}} \ge p/4 -1$ for any $\lambda_{-1}$.
	Thus, \eqref{Exp_J} is bounded above by
	\bea
	&&  \sum_{j=0}^k \Big(  \frac{k^2}{ p/4 - 1 -k }  \Big)^j    \Big[ \exp \big\{ -n \log ( 1- j \epsilon_{np}^2) \big\} \frac{3}{2}  -1\Big] \\
	&=& \frac{1}{2} + \sum_{j=1}^k \Big(  \frac{k^2}{ p/4 - 1 -k }  \Big)^j    \Big[ \exp \big\{ -n \log ( 1- j \epsilon_{np}^2) \big\} \frac{3}{2}  -1\Big]  \\
	&\le& \frac{1}{2} + \frac{3}{2} \sum_{j=1}^k  \Big(  \frac{k^2}{ p/4 - 1 -k }  \Big)^j p^{2 \nu^2 j }  \\
	&\le& \frac{1}{2} + \frac{3C}{2} \sum_{j=1}^k p^{-\epsilon j} p^{  (\epsilon/2) j }   \\
	&\le& c_2^2
	\eea
	for some constant $C>0$ and all sufficiently large $p$, by setting $c_2^2 = 3/4 <1$, where the second inequality follows from $s_0^2 = O(p^{3-\epsilon})$ and $\nu = \sqrt{\epsilon/4}$.
	This implies \eqref{div_c2}, which completes the proof of \eqref{lower1}.

	\paragraph{(ii) Proof of \eqref{lower2}} 
	Define a parameter space 
	\bea
	B_2  &:=& \Big\{  \sg(\theta):   \sg(\theta) = I_p +  \frac{\nu}{\sqrt{n}} diag(\theta) , \quad \theta\in\Theta = \{0,1\}^p   \Big\}
	\eea
	for some small constant $\nu>0$.
	Then, it is easy to see that $B_2 \subset \calU(s_0, \tau_0)$ for any $\tau_0 >1$ and any sufficiently large $n$.
	By the Assouad lemma in \cite{cai2012optimal}, we have
	\bea
	\inf_{\hat{\sg}} \max_{\sg(\theta)\in B_2 } 2^2 \bbE_{\theta} \| \hat{\sg} - \sg(\theta) \|_F^2 
	&\ge& \min_{H(\theta, \theta') \ge1 } \frac{\|\sg(\theta) - \sg(\theta') \|_F^2 }{H(\theta, \theta') } \, \frac{p}{2} \, \min_{H(\theta, \theta') =1} \| \bbP_\theta \wedge \bbP_{\theta'} \|  .
	\eea
	For any two probability measures $\bbP$ and $\mathbb{Q}$, let $\|\bbP - \mathbb{Q} \|_1 = \int |p-q| d\mu $, where $p$ and $q$ ar probability densities corresponding to $\bbP$ and $\mathbb{Q}$, respectively, with respect to a common dominating measure $\mu$.
	Since $\|\sg(\theta) - \sg(\theta') \|_F^2 = H(\theta,\theta')  \nu^2/n$ and $\| \bbP_\theta \wedge \bbP_{\theta'} \|  = 1 - \| \bbP_\theta - \bbP_{\theta'}\|_1/2$, it suffices to prove that 
	\bean\label{P_L1_ineq}
	\| \bbP_\theta - \bbP_{\theta'}\|_1 &\le& \frac{1}{2}
	\eean
	for any $\theta,\theta' \in \Theta$ such that  $H(\theta,\theta')=1$.
	
	By inequality (C.11) in \cite{lee2018optimal}, we have
	\bea
	\| \bbP_\theta - \bbP_{\theta'}\|_1
	&\le& n \, \Big[  tr\big\{  \sg(\theta') \sg(\theta)^{-1} \big\}  - \log \det  \big\{  \sg(\theta') \sg(\theta)^{-1} \big\} - p  \Big]  \\
	&\equiv&  n \, \Big[  tr \big\{  \sg(\theta)^{-1/2} D_1 \sg(\theta)^{-1/2}  \big\}  - \log \det \big\{  \sg(\theta)^{-1/2} D_1 \sg(\theta)^{-1/2} + I_p   \big\}    \Big]  ,
	\eea
	where $D_1 = \sg(\theta') - \sg(\theta)$.
	Then, by Lemma C.2 of \cite{lee2018optimal}, 
	\bea
	\| \bbP_\theta - \bbP_{\theta'}\|_1 
	&\le& n \, R  \\
	&\le& n \, c \| D_1 \sg(\theta)^{-1} \|_F^2  \,\,\le \,\, c \nu^2   
	\eea
	for some constant $c>0$ and any $\theta,\theta' \in \Theta$ such that  $H(\theta,\theta')=1$.
	Therefore, by taking $\nu^2 = 1/(2c)$, it shows that \eqref{P_L1_ineq} holds.  \hfill $\blacksquare$	
\end{proof}

\begin{lemma}\label{lem:Rdecomp}
	If $s_0^2 (\log p)^3 = O(p^2 n)$ and $s_0^2 = O(p^{3-\epsilon})$ for some small constant $\epsilon>0$, then  
	\bean\label{R_equality}
	R_{\lambda_1,\lambda_1'}^{\gamma_{-1},\lambda_{-1} } 
	&=& - 2 \log ( 1- J \epsilon_{np}^2 ) + R_{1,\lambda_1,\lambda_1'}^{\gamma_{-1},\lambda_{-1} } ,
	\eean
	where $R_{\lambda_1,\lambda_1'}^{\gamma_{-1},\lambda_{-1} } $ is defined in \eqref{Rlam}, and $R_{1,\lambda_1,\lambda_1'}^{\gamma_{-1},\lambda_{-1} } $ satisfies 
	\bean\label{EEupper}
	\bbE_{ (\lambda_1,\lambda_1' ) \mid J  } \left[  \bbE_{(\gamma_{-1}, \lambda_{-1}) \mid (\lambda_1,\lambda_1')} \Big\{   \exp \big( \frac{n}{2} R_{1,\lambda_1,\lambda_1'}^{\gamma_{-1},\lambda_{-1} }   \big)   \Big\}  \right]   &\le& \frac{3}{2}  
	\eean
	uniformly over all $J$.
\end{lemma}

\begin{proof}[Proof of Lemma \ref{lem:Rdecomp}]
	Note that Lemma \ref{lem:Rdecomp} corresponds to Lemma 11 of \cite{cai2012optimal}.
	Equation \eqref{R_equality} follows from equation (60) of \cite{cai2012optimal}.	
	However, to obtain \eqref{EEupper}, \cite{cai2012optimal} assumed $p \ge n^{\beta}$ for some $\beta>1$ and $n k \epsilon_{np}^3$ is sufficiently small for all large $n$.
	We will show that one can still prove that \eqref{EEupper} holds under the conditions in Lemma \ref{lem:Rdecomp}.
	
	Let $A_* = (I_p - \sg_0)(\sg_0-\sg_1) (\sg_0 - \sg_2)$, then by the same arguments used in pages 3-5 of Supplementary of \cite{cai2012optimal}, one can show that 
	\bea
	&& \bbE_{ (\lambda_1,\lambda_1' ) \mid J  } \left[  \bbE_{(\gamma_{-1}, \lambda_{-1}) \mid (\lambda_1,\lambda_1')} \Big\{   \exp \big( \frac{n}{2} R_{1,\lambda_1,\lambda_1'}^{\gamma_{-1},\lambda_{-1} }   \big)   \Big\}  \right]  \\
	&\le& \bbE_{ (\lambda_1,\lambda_1' ) \mid J  } \left[  \bbE_{(\gamma_{-1}, \lambda_{-1}) \mid (\lambda_1,\lambda_1')} \Big\{   \exp \big(  C n \, \max \{ \|A_*\|_1, \|A_* \|_\infty \}  \big)   \Big\}  \right]  
	\eea
	for some constant $C>0$, and  
	\bea
	\bbE_{ (\lambda_1,\lambda_1' ) \mid J  } \left(  \bbE_{(\gamma_{-1}, \lambda_{-1}) \mid (\lambda_1,\lambda_1')} \Big[   I \Big\{  \max (\|A_*\|_1 , \|A_*\|_\infty ) \ge 2 t \, k \epsilon_{np}^3   \Big\}  \Big] \right)
	&\le& 2p \Big(  \frac{k^2 }{p/8 -1 -k} \Big)^{t-1}
	\eea
	for every $t>2$.
	Thus,
	\bean
	&&  \bbE_{ (\lambda_1,\lambda_1' ) \mid J  } \left[  \bbE_{(\gamma_{-1}, \lambda_{-1}) \mid (\lambda_1,\lambda_1')} \Big\{   \exp \big(  C n \, \max \{ \|A_*\|_1, \|A_* \|_\infty \}  \big)   \Big\}  \right]  \nonumber \\
	&\le& a + \int_{x>a}  \bbE_{ (\lambda_1,\lambda_1' ) \mid J  } \left(  \bbE_{(\gamma_{-1}, \lambda_{-1}) \mid (\lambda_1,\lambda_1')} \Big[  I \Big\{  \exp \big(  C n \, \max \{ \|A_*\|_1, \|A_* \|_\infty \}  \big)  > x  \Big\}  \Big] \right)  dx \nonumber \\
	&\le& \exp \Big( \frac{1+2\epsilon}{\epsilon} \, 2 C n k\epsilon_{np}^3 \Big)
	+ \int_{t > (1+2\epsilon)/\epsilon } 2C n k \epsilon_{np}^3 \exp \Big(  2 C t n k\epsilon_{np}^3 \Big)  \, 2p \Big(  \frac{k^2 }{p/8 -1 -k} \Big)^{t-1} dt \nonumber\\
	&\le& \exp \Big( \frac{1+2\epsilon}{\epsilon} \, 2 C n k\epsilon_{np}^3 \Big) \label{upper_term1}\\
	&+&  \int_{t > (1+2\epsilon)/\epsilon }  \exp \Big\{  \log (2p) - (t-1) \log \frac{p/8 -1 -k}{k^2 }   + 2C (t+1) n k \epsilon_{np}^3 \Big\}  dt  ,\label{upper_term2}
	\eean
	where the second inequality follows by choosing $a = \exp \{ 2C n k\epsilon_{np}^3 (1+2\epsilon)/\epsilon \}$.
	Since $k = \lceil c_{np}/2 \rceil -1$, $c_{np} = \lceil s_0/p \rceil$, $\epsilon_{np} = \nu \sqrt{\log p /n}$ with $\nu = \sqrt{\epsilon/4}$ and we assume that $s_0^2(\log p)^3 = O(p^2 n)$, term \eqref{upper_term1} is less than $3/2$ for any sufficiently small $\epsilon>0$.
	Thus, we complete the proof if we show that term \eqref{upper_term2} is of order $o(1)$.
	Note that 
	\bea
	(t-1) \log \frac{p/8 -1 -k}{k^2 } 
	&\ge& \Big( 1+ \frac{1}{\epsilon} \Big)  \log \frac{p/8 -1 -k}{k^2 }  \\
	&\ge&  \Big( 1+ \frac{1}{\epsilon} \Big)  \log \frac{p^3/8 -p^2 -p s_0}{ s_0^2 } + C' \\
	&=& \Big( 1+ \frac{1}{\epsilon} \Big)  \log \Big\{ \frac{p^3}{s_0^2} \Big( \frac{1}{8} - \frac{1}{p} - \frac{s_0}{p^2}   \Big) \Big\}  +C' \\
	&\ge& \Big( 1+ \frac{1}{\epsilon} \Big)  \log ( p^\epsilon )  + C'' \\
	&=& (1+\epsilon) \log p + C'' ,
	\eea 
	for any $t>(1+2\epsilon)/\epsilon$ and some constnats $C'>0$ and $C''>0$.
	The third inequality follows from the assumption $s_0^2 = O(p^{3-\epsilon})$.
	Therefore, it implies that \eqref{upper_term2} is of order $o(1)$, which gives the desired result. 	\hfill $\blacksquare$
\end{proof}

\section*{Appendix D: Full Conditionals}
The joint posterior distribution of $\boldsymbol{\Sigma}$ and $\boldsymbol{\rho} = (\rho_{jk})$ with shrinkage priors \eqref{sig_jk}, \eqref{rho_jk}, and \eqref{sig_jj} is proportional to
$$
	|\boldsymbol{\Sigma}| ^{-n/2}\exp\left\{-\frac{1}{2}tr(S\boldsymbol{\Sigma}^{-1})\right\}\prod_{j < k}\left[\exp\left\{- \frac{\sigma_{jk}}{2\tau_1^2}\left(\frac{1 - \rho_{jk}}{\rho_{jk}}\right) \right\}\rho_{jk}^{a-1}(1 - \rho_{jk})^{b-1}\right]\prod_{j=1}^p\exp\left\{-\frac{\lambda}{2}\sigma_{jj} \right\},
$$
and under partitions \eqref{eq:partitions} and the transformation \eqref{eq:changev}, the joint conditional posterior of $\boldsymbol{u}$ and $v$ given $\boldsymbol{\rho}$ \citep{wang15} is
\begin{eqnarray*}
	\pi\left(\boldsymbol{u}, v \mid \boldsymbol{\rho}, \boldsymbol{X}_n \right) \propto \exp\Bigg\{-\frac{1}{2}\big(n\log(v) + \boldsymbol{u}^\top \boldsymbol{\Sigma}_{11}^{-1}\boldsymbol{S}_{11}\boldsymbol{\Sigma}_{11}^{-1}\boldsymbol{u}v^{-1} - 2\boldsymbol{s}_{12}^\top\boldsymbol{\Sigma}_{11}^{-1}\boldsymbol{u}v^{-1} + s_{22}v^{-1} &&\\
	 + \boldsymbol{u}^\top \boldsymbol{D}^{-1}\boldsymbol{u} + \lambda \boldsymbol{u}^\top \boldsymbol{\Sigma}_{11}^{-1}\boldsymbol{u} + \lambda v\big)\Bigg\}, &&
\end{eqnarray*}
where $\boldsymbol{D} = \diag(\boldsymbol{v}_{12})$. 

\noindent This gives the full conditional posteriors of $\boldsymbol{u}$ and $v$ as follows \citep{wang15}:
\begin{eqnarray*}
	\pi(\boldsymbol{u} \mid v, \boldsymbol{\rho}, \boldsymbol{X}_n) &=& N_{p-1} \left[\left\{\bsB + \boldsymbol{D}^{-1}\right\}^{-1}\bsw, \left\{\bsB + \boldsymbol{D}^{-1}\right\}^{-1}\right], \\
	 \pi(v \mid \boldsymbol{u}, \boldsymbol{\rho}, \boldsymbol{X}_n) &=& GIG \left(1 - n/2, \,\, \lambda, \,\, \bsu^T\Sigma_{11}^{-1}\bsS_{11}\Sigma_{11}^{-1}\bsu - 2s_{12}^T\Sigma_{11}^{-1}\bsu + s_{22}\right),
\end{eqnarray*}
where $\bsB = \bsSig_{11}^{-1}\bsS_{11}\bsSig_{11}^{-1}v^{-1} + \lambda\bsSig_{11}^{-1}$ and $\bsw = \bsSig_{11}^{-1}\bss_{12}v^{-1}$.

\noindent Finally, to dervie the full conditional of $\boldsymbol{\rho}$, we consider a reparametrization of $\rho_{jk}$ as 
$$\phi_{jk} = \frac{\rho_{jk}}{1 - \rho_{jk}},$$
then the shrinkage prior can be represented as follows \citep{armagan11}:
$$
	\sigma_{jk} \mid \phi_{jk} \sim N(0, \phi_{jk}\tau_1^2), \,\,\, \phi^{1/2}_{jk} \sim C^+(0, 1),
$$
where $C^+(0, 1)$ denotes a half-Cauchy distribution on $(0, \infty)$. The full conditional distribution of $\phi_{jk}$ with an additional parameter $\psi_{jk}$ \citep{carvalho2010horseshoe} is given as
\begin{eqnarray*}
	\pi(\psi_{jk} \mid \phi_{jk}, \bsSig, \bsX_n) &=& Gamma(a + b, \phi_{jk} + 1), \\
	\pi(\phi_{jk} \mid \psi_{jk}, \bsSig, \bsX_n) &=& GIG(a - 1/2, 2\psi_{jk}, \sigma_{jk}^2/\tau_1^2).
\end{eqnarray*}

\bibliographystyle{elsarticle-harv}
\bibliography{sparseCOV}

\end{document}